\pdfoutput=1
\documentclass[11pt, reqno]{amsart}
\usepackage{amsmath, amssymb, amsthm, verbatim,enumerate,bbm, mathtools,color,enumitem, tikz, graphicx, commath}
\usepackage{subcaption}
\usetikzlibrary{positioning,chains,fit,shapes,calc,decorations.markings}
\usepackage{dsfont}
\usepackage{hyperref}
\usepackage{lineno}
\usepackage{xcolor}
\usepackage{fullpage}

\def\V#1{\mathbf{#1}}

\def\vu{{\V u}}
\def\vv{{\V v}}

\def\vx{{\V x}}
\def\vy{{\V y}}

\def\ip#1{\left\langle #1\right\rangle}

\DeclareMathOperator{\diag}{diag}
\DeclareMathOperator{\rank}{rank}
\def \tran {\mathsf{T}}
\DeclareMathOperator{\SSubset}{Subset}

\renewcommand{\Pr}[2][]{\mathbb{P}_{#1} \left[ #2 \rule{0mm}{3mm}\right]}

\newtheorem{thm}{{Theorem}}[section]

\newtheorem{cor}[thm]{{Corollary}}
\newtheorem{prop}[thm]{{Proposition}}
\newtheorem{lem}[thm]{{Lemma}}
\newtheorem{defi}[thm]{{Definition}}

\newtheorem{claim}[thm]{{Claim}}
\newtheorem{conjecture}[thm]{Conjecture}
\newtheorem{obs}[thm]{Observation}

\makeatletter
\newtheorem*{rep@theorem}{\rep@title}
\newcommand{\newreptheorem}[2]{%
\newenvironment{rep#1}[1]{%
 \def\rep@title{#2 \ref{##1}}%
 \begin{rep@theorem}}%
 {\end{rep@theorem}}}
\makeatother

\newreptheorem{thm}{Theorem}
\newreptheorem{lem}{Lemma}
\newreptheorem{claim}{Claim}

\newcommand{\var}{\varepsilon}
\newcommand{\R}{\mathbb{R}}
\newcommand{\E}{\mathbb{E}}

\title{Hamiltonicity of sparse pseudorandom graphs}
\author{Asaf Ferber}
\address{Department of Mathematics, University of California, Irvine.
Email:  {{asaff@uci.edu}}. }
\thanks{A.F. is partially supported by NSF grant DMS-1953799, NSF Career DMS-2146406, a Sloan's fellowship, and an Air force grant FA9550-23-1-0298. J.H.~is partially supported by Natural Science Foundation of China (12371341). R.V. is partially supported by NSF grant DMS-1954233, NSF grant DMS-2027299, U.S. Army grant 76649-CS, and NSF+Simons Research Collaborations on the Mathematical and Scientific Foundations of Deep Learning.} 
\author{Jie Han}
\address{School of Mathematics and Statistics and Center for Applied Mathematics, Beijing Institute of Technology, Beijing, China. Email: {{han.jie@bit.edu.cn}}} 
\author{Dingjia Mao}
\address{Department of Mathematics, University of California, Irvine.
Email:  {{dingjiam@uci.edu}}}
\author{Roman Vershynin}
\address{Department of Mathematics, University of California, Irvine.
Email:  {{rvershyn@uci.edu }}}
\date{\today}

\begin{document}
\maketitle

%\linenumbers
%%%%%%%%%%%%%%%%%%%%%%%%%%%%%%%%%%%%%%%%%%%%%%%%%%%%%%%%%%%%%%%%%%%%%%

\begin{abstract}
%Studying the class of   $(n,d,\lambda)$-graphs is one the classical topic in graph theory. 
We show that every $(n,d,\lambda)$-graph contains a Hamilton cycle for sufficiently large $n$, assuming that $d\geq \log^{6}n$ and $\lambda\leq  cd$, where $c=\frac{1}{70000}$. This significantly improves a recent result of Glock, Correia and Sudakov, who obtained a similar result for $d$ that grows polynomially with $n$. The proof is based on a new result regarding the second largest eigenvalue of the adjacency matrix of a subgraph induced by a random subset of vertices, combined with a recent result on connecting designated pairs of vertices by vertex-disjoint paths in $(n,d,\lambda)$-graphs. We believe that the former result is of independent interest and will have further applications. 
\end{abstract}

\section{Introduction}\label{s: intro}

A \emph{Hamilton cycle} in a graph is a cycle that passes through all the vertices of the
graph exactly once, and a graph  containing a Hamilton cycle is called  \emph{Hamiltonian}. Even though a Hamilton cycle is a relatively simple structure, determining whether a certain graph is Hamiltonian was included in the list of $21$ NP-hard problems by Karp \cite{karp1972reducibility}. Thus, there is significant interest in deriving conditions that ensure Hamiltonicity in a given graph. 
For instance, the celebrated Dirac's theorem~\cite{dirac1952some} states that every 
graph on $n \geq 3$ vertices with a minimum degree of $n/2$ is Hamiltonian. For more results on Hamiltonicity, readers can refer to %e.g., %\cite{ajtai1985first,ferber2023counting,ferber2016rainbow} and 
the surveys \cite{frieze2019hamilton,kuhn2012survey,kuhn2014hamilton}.

Most classical sufficient conditions for a graph to be Hamiltonian are only applicable to relatively dense graphs, such as those considered in Dirac's Theorem. Establishing sufficient conditions for Hamiltonicity in sparse graphs is known to be much more challenging. Sparse random graphs are natural objects to consider as starting points, and they have attracted a lot of attention in the past few decades. 
In 1976, P\'osa \cite{posa1976hamiltonian} proved that for some large constant $C$, the binomial random graph model $G(n, p)$ with
$p \geq C\log n/n$ is typically Hamiltonian. In the following few years, Korshunov \cite{korshunov1976solution} refined P\'osa’s result, and in 1983, Bollob\'as \cite{bollobas1984evolution}, and independently Koml\'os and Szemer\'edi  \cite{komlos1983limit} showed a more precise threshold for Hamiltonicity. Their results demonstrate 
that if $p = (\log n + \log \log n + \omega(1))/n$, then the probability of the random graph $G(n, p)$ being Hamiltonian tends
to 1 (we say such an event happens \emph{with high probability}, or \emph{whp} for brevity).

Following the fruitful study of random graphs, it is natural to explore families of deterministic graphs that behave in some ways like random graphs; these are sometimes called \emph{pseudorandom graphs}. A natural candidate to begin with is the following: suppose that we sample a random graph $G\sim G(n,p)$, and then allow an adversary to delete a constant fraction of the edges incident to each vertex. The resulting subgraph $H\subseteq G$ loses all its randomness. Thus, we cannot use, for example, a multiple exposure trick and concentration inequalities, which were heavily used in the proof of Hamiltonicity of a typical $G\sim G(n,p)$. Under such a model, one of the central problems to consider is quantifying the \emph{local resilience} of the random graph $G$ with respect to Hamiltonicity. In \cite{sudakov2008local}, Sudakov and Vu initiated the study of local resilience of random graphs, and they showed that for any $\var > 0$, if $p$ is somewhat greater than $\log ^4n/n$, then $G(n, p)$ typically has the property that every spanning subgraph
with a minimum degree of at least $(1 + \var)np/2$ contains a Hamilton cycle. They also conjectured that this remains true as long as $p=(\log n+\omega(1))/n$, which was solved by Lee and Sudakov \cite{lee2012dirac}. Later, an even stronger result, the so-called ``hitting-time” statement, was shown by Nenadov, Steger and Truji\'c \cite{nenadov2019resilience}, and Montgomery \cite{montgomery2019hamiltonicity}, independently.

Exploring the properties of pseudorandom graphs, which has attracted many researchers in the area, is much more challenging than studying random graphs. The first quantitative notion of pseudorandom graphs was introduced by Thomason \cite{thomason1987pseudo,thomason1987random}. He initiated the study of pseudorandom graphs by introducing the so-called \emph{$(p,\lambda)$-jumbled graphs}, 
which satisfy $|e(U)-p\binom{|U|}{2}|\leq \lambda|U|$ for every vertex subset 
$U\subseteq V$. Since then, there has been a great deal of investigation into different types and various properties  of pseudorandom graphs, for example, \cite{allen2014powers,conlon2014extremal,han2022factors,han2021finding,kohayakawa2007turan,nenadov2019triangle}. This remains a very active area of research in graph theory. 

One special class of pseudorandom graphs which has been studied extensively is the class of \emph{spectral expander graphs}, also known as \emph{$(n,d,\lambda)$-graphs}. Given a graph $G$ on vertex set $V=\{v_1,\ldots,v_n\}$, its \emph{adjacency matrix} $A:=A(G)$ is an $n\times n$, $0/1$ matrix, defined by $A_{ij}=1$ if and only if $v_iv_j\in E(G)$. Let $s_1(A)\geq s_2(A)\geq\cdots \geq s_n(A)$ be the \emph{singular values} of $A$ (see Definition \ref{def:singular-value}). Observe that for a $d$-regular graph $G$, we always have $s_1(G):=s_1(A(G))=d$, so the largest singular value is not a very interesting quantity. We say that $G$ is an \emph{$(n,d,\lambda)$-graph} if it is a $d$-regular graph on $n$ vertices with $s_2(G) \leq \lambda$. 

The celebrated Expander Mixing Lemma (see, e.g. Chapter 9 in \cite{alon2016probabilistic}) provides a
powerful formula to estimate the edge distribution of an $(n, d, \lambda)$-graph, which suggests that $(n,d,\lambda)$-graphs are indeed special cases of jumbled graphs, and that $G$ has stronger expansion properties for smaller values of $\lambda$. Thus, it is natural to seek for the best possible condition on the \emph{spectral gap} (defined as the ratio $\lambda/d$) which guarantees certain properties. Examples of such results can be found e.g. in \cite{alon2007embedding,balogh2010large,han2022spanning,pavez2023spanning}. For more on $(n,d,\lambda)$-graphs and their many applications, we refer the reader to the surveys of Hoory, Linial and Wigderson \cite{hoory2006expander}, Krivelevich and Sudakov \cite{krivelevich2006pseudo}, the book of Brouwer and Haemers \cite{brouwer2011spectra}, and the references therein.

Hamiltonicity of $(n,d,\lambda)$-graphs was first studied by Krivelevich and Sudakov \cite{krivelevich2003sparse}, who proved a sufficient condition on the spectral gap forcing Hamiltonicity. 
More precisely, they showed that for sufficiently large $n$, any $(n,d,\lambda)$-graph with
\[
\lambda/d\leq\frac{(\log\log n)^2}{1000\log n(\log\log\log n)}
\]
has a Hamilton cycle. In the same paper, Krivelevich and Sudakov made the following conjecture.

\begin{conjecture}\label{conj}
There exists an absolute constant $c > 0$ such that for any sufficiently large integer $n$, any $(n,d,\lambda)$-graph with $\lambda/d\leq c$ contains a Hamilton cycle.
\end{conjecture}

Although there are numerous related results in this direction, there had been no improvement on the original bound until the recent result given by Glock, Correia and Sudakov \cite{glock2024hamilton}. In their paper, they improved the above result in two different ways: $(i)$ they demonstrated that the spectral gap $\lambda/d\leq c/(\log n)^{1/3}$ already guarantees Hamiltonicity; $(ii)$ they confirmed Conjecture \ref{conj} in the case where $d\geq n^{\alpha}$ for every fixed constant $\alpha>0$.

In this paper, we improve the second result in \cite{glock2024hamilton}.

\begin{thm}\label{main-thm}
There exists an absolute constant $c > 0$ such that
for any sufficiently large integer $n$, 
any $(n,d,\lambda)$-graph 
with $\lambda/d\leq c$ and $d \geq \log^{6}n$
contains a Hamilton cycle.
\end{thm}

Our proof works for $c=\frac{1}{70000}$, although we made no attempt to optimize this constant.

It is worth mentioning that Dragani\'c, Montgomery, Correia, Pokrovskiy and Sudakov independently verified Conjecture \ref{conj} in \cite{draganic2024hamiltonicity}, and in particular, they proved a stronger statement than our main result. Their approach relies on extensions of the P\'osa rotation-extension technique and sorting networks. In contrast, while our work utilizes a previous result on closing vertex-disjoint paths into a cycle, it is primarily based on new machinery introduced in this paper, as summarized in Theorem \ref{thm:subgraphs-of-pseudorandom-graphs}. Specifically, we show that the spectral gap of a random induced subgraph of an $(n,d,\lambda)$-graph is typically bounded above by the spectral gap of the original graph, up to a constant factor.

To achieve this, we utilize results on norms of principal submatrices, such as the Rudelson-Vershynin Theorem \cite{rudelson2007sampling} (see Section \ref{s: random subgraphs}), and demonstrate that, with probability $1 - n^{-\Theta(1)}$, the spectral gap of the induced subgraph remains $O(\lambda/d)$ for a sufficiently large random vertex subset. We believe that this result will have further applications.

The paper is organized as follows. In Section \ref{s: outline}, we provide an outline of the proof. Section \ref{s: EML matrices} contains the proof of the expander mixing lemma for matrices, followed by an analysis of the special case for \emph{almost $(n,d,\lambda)$-graphs} in Section \ref{s: expanders}. In Section \ref{s: extendability}, we introduce the extendability property and reference a useful result from \cite{hyde2023spanning}, which ensures that vertex-disjoint paths can be used to connect designated pairs of vertices in $(n,d,\lambda)$-graphs. Our key lemma, which concerns the second singular value of a random induced subgraph of an $(n,d,\lambda)$-graph, is presented in Section \ref{s: random subgraphs}. Finally, in Section \ref{s: main proof}, we prove our main result, Theorem \ref{main-thm}, along with a generalized version, Theorem \ref{generalized-main-thm}, for ``almost'' $(n,d,\lambda)$-graphs. For the reader's convenience, we also include some standard tools from linear algebra and several technical proofs in the Appendix.

\subsection{Notation}

%\subsection*{Graphs}

For a graph $G=(V,E)$, let $e(G):=|E(G)|$. We mostly assume that $V=[n]$ for simplicity. For a subset $A\subseteq V$ of size $m$, we simply call it an \emph{$m$-set}, and we denote the family of all $m$-sets of $V$ by $\binom{V}{m}$. For two vertex sets $A, B \subseteq V (G)$, we define $E_G(A,B)$ to be the set of all edges $xy\in E(G)$ with $x\in A$ and $y\in B$, and set $e_G(A,B):=|E_G(A,B)|$. For two disjoint subsets $X,Y\subseteq V$, we write $G[X,Y]$ to denote the induced bipartite subgraph of $G$ with parts $X$ and $Y$. Moreover, we define $N_G(v)$ to be the neighborhood of a vertex $v$, and define $N_G(A) := \bigcup_{v\in A} N_G(v) \setminus A$ for a subset $A\subseteq V$. We write $N_G(A, B) = N_G(A) \cap B$ and for a vertex $v$, let
$N_G(v, B) = N_G(v) \cap B$. We also write $\deg_G(v):=|N_G(v)|$ and $\deg_G(v,B):=|N_G(v,B)|$. Finally, let $\delta(G)$ be the minimum degree of $G$ and let $\Delta(G)$ be the maximum degree of $G$.

The \emph{adjacency matrix} of $G$, denoted by $A:=A(G)$, is a $0/1$, $n\times n$ matrix such that $A_{i,j}=1$ if and only if $ij\in E(G)$. Moreover, given any subset $X\subseteq V$, its characteristic vector $\mathbbm{1}_X\in \mathbb{R}^n$ is defined by 
\[\mathbbm{1}_X(i)=\begin{cases}
    1 &\textrm{ if } i\in X\\
    0 &\textrm{ otherwise}
\end{cases}.\]

We will often omit the subscript to ease the notation, unless otherwise stated. 
Since all of our calculations are asymptotic, we will often omit floor and ceiling functions whenever they are not crucial.

\section{Proof Outline}\label{s: outline}

Our strategy for finding a Hamilton cycle in an $(n,d,\lambda)$-graph $G$ consists of two main phases. %First, we pick two disjoint vertex subsets $X,Y\subseteq V(G)$ of the same size $\Theta(n/\log^4n)$, where a result in \cite{hyde2023spanning} helps form a subgraph $S_{res}\subseteq G$ with $|V(S_{res})|=\Theta(n/\log n)$, covering $X$ and $Y$. 
First, taking two disjoint vertex subsets $X,Y\subseteq V(G)$ of the same size $\Theta(n/\log^4n)$, we find a subgraph $S_{res}\subseteq G$ with $|V(S_{res})|=\Theta(n/\log n)$ covering $X$ and $Y$, using a recent result (see Lemma \ref{lemma:find_sorting_network_in_expander} later) of Hyde, Morrison, M\"uyesser and Pavez-Sign\'e~\cite{hyde2023spanning}. This subgraph includes various path factors for later use, where each path has one endpoint in $X$ and the other in $Y$. Then, we cover $V(G)\setminus (V(S_{res})\setminus (X\cup Y))$ by vertex-disjoint paths, with one endpoint in $X$ and the other in $Y$. Now, we are allowed to close the paths into a cycle by using one path factor in the prepared subgraph $S_{res}$. Since all the vertices are used and passed through exactly once, the cycle is indeed a Hamilton cycle.

We now explain our method thoroughly. First, we take two random disjoint subsets $X,Y\subseteq V(G)$ of equal size $\Theta(n/\log^4 n)$. Using Proposition \ref{prop:weakerextendability}, we can deduce that whp, the empty graph $I(X\cup Y)$ is ``extendable" (see Definition \ref{def:dmextendable}), which further produces a crucial subgraph $S_{res}\subseteq G$ on $\Theta(n/\log n)$ vertices such that $X\cup Y\subseteq V(S_{res})$ (see Lemma \ref{lemma:find_sorting_network_in_expander}). The powerful property of $S_{res}$ that we will use is the following: for any ordering of the pairs in $(X,Y)$, there exists a path factor in $S_{res}$ connecting such pairs. This property will be used to connect the paths with endpoints in $X$ and $Y$ obtained in the second phase.

Next, since $|V(S_{res})|=\Theta(n/\log n)$, we can utilize its randomness in a way so that after removing it, the graph is still pseudorandom. Thus, by randomly partitioning $V(G)\setminus V(S_{res})$ into $|X|$-sets, if we can find a perfect matching between each two consecutive parts, we will obtain the desired vertex-disjoint paths $P_i$ connecting $x_i\in X$ and $y_i\in Y$. Now, using the path factor in $S_{res}$ connecting $(x_i,y_i)$s, we can concatenate all the paths $P_i$ into a cycle. 

It remains to ensure, whp, perfect matchings between two random disjoint subsets in an expander graph. To prove this, we demonstrate that the bipartite subgraph induced by each two consecutive parts is a good expander. 
Equivalently, it suffices to study the spectral properties of random induced subgraphs of $G$, and this is the main contribution of this paper. It is crucial to remark that although there are some previous results on randomly selecting edges, e.g. \cite{chung2007spectral}, we randomly pick vertex subsets instead of picking edges. 
Using results on norms of principal matrices, e.g.~Rudelson-Vershynin theorem in \cite{rudelson2007sampling}, we show that with probability at least $1-n^{-\Theta(1)}$, the spectral gap of a random induced subgraph of $(n,d,\lambda)$-graph is still $O(\lambda/d)$
 (see Theorem \ref{thm:subgraphs-of-pseudorandom-graphs}).

\section{Expander mixing lemma for matrices}\label{s: EML matrices}
%============

One of the most useful tools in spectral graph theory is the expander mixing lemma, which asserts that an $(n,d,\lambda)$-graph is an expander (see, e.g., \cite{hoory2006expander}).

\begin{thm}[Expander mixing lemma]\label{expander-mixing-lemma}
Let $G=(V,E)$ be an $(n,d,\lambda)$-graph. Then, for any two subsets $S,T\subseteq V$, we have
\[
\left|e(S,T)-\frac{d|S||T|}{n}\right|
\leq \lambda \sqrt{\abs{S} \left( 1-\frac{\abs{S}}{n} \right) \, \abs{T} \left( 1-\frac{\abs{T}}{n} \right)}.
\]
\end{thm}

We will need a more general version of the expander mixing lemma which can be applied to non-regular graphs, digraphs, and even to general $m \times n$ matrices $A$. To state such a general result, it is convenient to normalize $A$ in the following way:

\begin{defi}[Normalized matrix] \label{def:N(A)} 
    Let $A$ be an $m\times n$ matrix. Let $L=L(A)$ be the $m\times m$ diagonal matrix with $L_{i,i}=\sum_{j}A_{i,j}$ for all $i$ (that is, the sum of entries in the $i$th row), and $R=R(A)$ be the $n\times n$ diagonal matrix with $R_{j,j}=\sum_{i}A_{i,j}$ (that is, the sum of entries in the $j$th column). The \emph{normalized matrix} of the matrix $A$ is defined as
    $$
    \bar{A}:=L^{-1/2}AR^{-1/2}.
    $$ 
    In particular, if $A$ is a symmetric $n \times n$ matrix, then the diagonal matrix $L(A)=R(A)\eqqcolon D(A)$ is called the {\em degree matrix} of $A$.
\end{defi}

Since the notion of eigenvalues is undefined for non-square matrices, it would be convenient for us to work with \emph{singular values} which are defined as follows for all matrices.

\begin{defi}[Singular values] \label{def:singular-value}
Let $A$ be a real $m\times n$ matrix. The \emph{singular values} of $A$ are the nonnegative square roots of the eigenvalues of the symmetric positive semidefinite matrix $A^\tran A$. We will always assume that $s_k(A)$ is the $k$th singular value of $A$ in nonincreasing order. In particular, the singular values and the eigenvalues of a symmetric positive semidefinite matrix $A$ coincide.
\end{defi}

 We are now ready to state a more general version of the expander mixing lemma. 

\begin{thm}[Expander mixing lemma for %non-symmetric
matrices] \label{thm:EMLM}
 Let $A$ be an $m\times n$ matrix with nonnegative entries, and let $\bar{A}$ be the normalized matrix of $A$. Then, for any two subsets $S\subseteq [m]$ and $T\subseteq [n]$, we have
 \[\left|A(S,T)-\frac{A(S,n) \, A(m,T)}{A(m,n)}\right|
 \leq s_2(\bar{A})\sqrt{A(S,n)\left(1-\frac{A(S,n)}{A(m,n)}\right) \, A(m,T) \left(1-\frac{A(m,T)}{A(m,n)}\right)},\]
 where we adopt the notation $A(S,T) \coloneqq \sum_{i \in S, j \in T} A_{i,j}$, 
 and we abbreviate $A(S,n) \coloneqq A(S,[n])$, $A(m,T) \coloneqq A([m],T)$, 
 and $A(m,n) \coloneqq A([m],[n])$. 
\end{thm}

Observe that Theorem~\ref{thm:EMLM} trivially implies Theorem~\ref{expander-mixing-lemma}, since the adjacency matrix of a $d$-regular graph satisfies 
\[
\bar{A}=\frac{1}{d}A.
\]

The proof of Theorem~\ref{thm:EMLM} is almost identical to the standard proof of Theorem~\ref{expander-mixing-lemma} that can be found e.g. as Proposition 4.3.2 in \cite{brouwer2011spectra}. Since we could not find a reference for this specific statement and its proof, we include the proof of Theorem~\ref{thm:EMLM} for the convenience of the reader, without claiming any originality. It is based on the following crucial observation.

\begin{obs} \label{obs: |N|}
    Let $A$ be an $m \times n$ matrix with nonnegative entries.
    Let $a \coloneqq A(m,n)$ and let $\mathbbm{1}_n$ denote the vector in $\R^n$ whose all coordinates are equal to $1$. Consider the vectors $\vu_1 \coloneqq a^{-1/2} L^{1/2} \mathbbm{1}_m$ and $\vv_1 \coloneqq a^{-1/2} R^{1/2} \mathbbm{1}_n$. Then: 
    \begin{enumerate}
        \item both $\vu_1$ and $\vv_1$ are unit vectors;
        \item $\bar{A} \vv_1=\vu_1$;
        \item $s_1(\bar{A}) = \norm{\bar{A}} = \vu_1^\tran \bar{A} \vv_1 = 1$.
    \end{enumerate}
\end{obs}

\begin{proof} 
The first two parts readily follow from the definitions of $a$, $L$, $R$, and $\bar{A}$.
As for the third part, the equation $s_1(\bar{A}) = \norm{\bar{A}}$ holds for any matrix. Let us show that $\|\bar{A}\|\leq 1$. For every $\|\vx\|_2=\|\vy\|_2=1$, we have
\[0\leq \sum_{i\in [m], j \in[n]}A_{i,j}\left(\frac{x_i}{\sqrt{L_{i,i}}}-\frac{y_j}{\sqrt{R_{j,j}}}\right)^2=2-2\sum_{i\in [m], j \in[n]}\frac{A_{i,j}x_iy_j}{\sqrt{L_{i,i}R_{j,j}}}=2-2\vx^\tran \bar{A} \vy.\]
This implies that $\vx^\tran \bar{A} \vy\leq 1$ for all unit vectors $\vx$ and $\vy$, which yields $\|\bar{A}\|\leq 1$.

Moreover, by definition of $\bar{A}$, we have $\vu_1^\tran \bar{A} \vv_1=1$. Therefore, by definition of the operator norm, it follows that $\norm{\bar{A}} \ge 1$.
The observation is proved.
\end{proof}

\medskip

Now we are ready to prove Theorem \ref{thm:EMLM}.

\begin{proof} [Proof of Theorem \ref{thm:EMLM}]
Let $r=\rank(\bar{A})$, and let $1=s_1 \geq s_2 \geq \ldots \geq s_r> 0$ be all the positive singular values of $\bar{A}$ in nonincreasing order. 
Applying the singular value decomposition theorem (Theorem \ref{thm:SVD}) combined with Observation~\ref{obs: |N|}, we can find orthonormal bases $\{\vu_1,\ldots,\vu_m\}$ of $\mathbb{R}^m$ and $\{\vv_1,\ldots,\vv_n\}$ of $\mathbb{R}^n$ with vectors $\vv_1$ and $\vu_1$ defined in Observation~\ref{obs: |N|}, and such that
\[
\bar{A}=\sum_{j=1}^r s_j \vu_j \vv_j^\tran.
\] 
In particular, $\bar{A}\vv_j=s_j\vu_j$ for $j=1,\ldots,r$ and $\bar{A}\vv_j=\mathbf{0}$ for $j>r$. Now, let $S\subseteq [m]$ and $T\subseteq [n]$ be two arbitrary subsets. Then 
\[
A(S,T) = \mathbbm{1}_S^\tran A \mathbbm{1}_T
= \chi_S^\tran \bar{A} \chi_T,
\quad \text{where} \quad
\chi_S \coloneqq L^{1/2}\mathbbm{1}_S, \quad 
\chi_T \coloneqq R^{1/2}\mathbbm{1}_T.
\]
Expanding both vectors as
\[\chi_S=\sum_{j=1}^m a_j\vu_j, \textrm{ and } \chi_T=\sum_{j=1}^n b_j\vv_j,\]
we obtain 
\[
    A(S,T) = \sum_{j=1}^r s_j a_j b_j = a_1 b_1 + \sum_{j=2}^r s_j a_j b_j.
\]
Recall from Observation~\ref{obs: |N|} that all singular values of $\bar{A}$ are bounded by $1$, and $r = \rank(\bar{A}) \le \min\{m,n\}$. 
Thus, by Cauchy-Schwarz inequality, we have 
\begin{equation}    \label{AST}
\abs{A(S,T) - a_1 b_1 } \le \sum_{j=2}^r \abs{a_j b_j} 
\le \left( \sum_{j=2}^m a_j^2 \right)^{1/2} \left( \sum_{j=2}^n b_j^2 \right)^{1/2}.
\end{equation}

Now observe that 
$a_1=\ip{\chi_S,\vu_1} = a^{-1/2} A(S,n)$ and 
$b_1=\ip{\chi_T,\vv_1} = a^{-1/2} A(m,T)$, so
\[
a_1 b_1 = \frac{A(S,n) A(m,T)}{a}. 
\]
Moreover, 
\[
\sum_{j=2}^m a_j^2 
= \norm{\chi_S}_2^2 - a_1^2
= A(S,n) - \frac{A(S,n)^2}{a}
= A(S,n)\left(1-\frac{A(S,n)}{a}\right),
\]
and similarly 
\[
\sum_{j=2}^n b_j^2 
= A(m,T) \left(1-\frac{A(m,T)}{a}\right).
\]
Substitute the last three identities into \eqref{AST} to complete the proof.
\end{proof}

\section{Almost regular expanders}\label{s: expanders}
%==========

Our argument relies on some spectral properties of random subgraphs of $(n,d,\lambda)$-graphs. Since random subgraphs are not expected to be {\em exactly} regular, we extend the definition of $(n,d,\lambda)$-graphs as follows: 

\begin{defi}[Almost $(n,d,\lambda)$-graphs]\label{def:pseudorandom}
    Let $d,\lambda>0$ and $\gamma\in[0,1)$. We say that a graph $G$ is an \emph{$(n, (1\pm\gamma)d, \lambda)$-graph} if\footnote{In this definition and elsewhere in the paper, we write $a = b \pm c$ as a shorthand for the double-sided inequality $b-c \le a \le b+c$. We use other similar abbreviations, whose exact meaning should be clear from context.} $G$ is a graph on $n$ vertices whose all degrees are $(1\pm\gamma)d$ and the second singular value of the adjacency matrix $A$ of $G$ satisfies $s_2(A) \le \lambda$.
\end{defi}

Almost $(n,d,\lambda)$-graphs behave similar to exact $(n,d,\lambda)$-graphs in many ways. If $G$ is an ({\em exactly}) $d$-regular graph with adjacency matrix $A$, its normalized adjacency matrix is obviously
\[
\bar{A}=\frac{1}{d}A
\]
according to Definition~\ref{def:N(A)}. If $G$ is an {\em almost} $d$-regular graph, its degree matrix $D=\diag(d_1,\ldots,d_n)$ is close to $dI$, and we can expect that
\[
\bar{A} = D^{-1/2} A D^{-1/2} \approx \frac{1}{d}A
\]
in some sense.
Below we show that such an approximation indeed holds in the sense of the closeness of all singular values. 

\begin{cor}[Singular values of almost regular graphs]    \label{cor: singular values normalized}
    Let $\gamma \in [0,1)$ and $d>0$. Let $G$ be a graph whose all vertices have degrees $(1\pm\gamma)d$. Then the adjacency matrix $A$ and the normalized adjacency matrix $\bar{A}$ of the graph $G$ satisfy
    \[
    \frac{s_k(A)}{(1+\gamma)d} \le s_k(\bar{A}) \le \frac{s_k(A)}{(1-\gamma)d} 
    \quad \text{for all } k \in [n].
    \]
\end{cor}

\begin{proof}
    Using the chain rule for singular values (Lemma~\ref{lem: chain rule}), we obtain
    \[
    s_k(A) = s_k \left( D^{1/2}\bar{A}D^{1/2} \right) \le \norm[1]{D^{1/2}}^2 s_k(\bar{A}).
    \]
    Since $\norm[1]{D^{1/2}}^2 = \norm{D} = \max_i d_i \le (1+\gamma)d$, the lower bound in Corollary~\ref{cor: singular values normalized} follows. The upper bound can be proved similarly. 
\end{proof}

\subsection{Expander mixing lemma for almost regular expanders}\label{s: EML graphs}
%---------------

Let us specialize Theorem~\ref{thm:EMLM} for almost $(n,d,\lambda)$-graphs. 

\begin{cor}[Expander mixing lemma for almost $(n,d,\lambda)$-graphs]\label{cor:EML-almost-regular}
    Let $G$ be an $(n,(1\pm\gamma)d,\lambda)$-graph. Then, for any two subsets $S,T\subseteq V(G)$, we have
    \begin{equation}    \label{eq: EML-almost-regular}
        \frac{(1-\gamma)^2d|S||T|}{(1+\gamma)n}-\var
        \leq e(S,T)
        \leq \frac{(1+\gamma)^2d|S||T|}{(1-\gamma)n}+\var,   
    \end{equation}
    where 
    \[
    \var=\frac{1+\gamma}{1-\gamma}\cdot \lambda \sqrt{|S||T|}. 
    \]
\end{cor}

\begin{proof}
Let $A$ and $\bar{A}$ be the adjacency and the normalized adjacency matrices of $G$, respectively. Theorem~\ref{thm:EMLM} yields
\begin{equation}   \label{AST simplified}
    \left|A(S,T)-\frac{A(S,n) \, A(n,T)}{A(n,n)}\right|
    \leq s_2(\bar{A})\sqrt{A(S,n) A(n,T)}.
\end{equation}
By Corollary~\ref{cor: singular values normalized} and assumption, we have 
\[
s_2(\bar{A}) \le \frac{s_2(A)}{(1-\gamma)d} \le \frac{\lambda}{(1-\gamma)d}.
\]
Moreover, since $A$ is an adjacency matrix, we have
$A(S,T)=e(S,T)$, 
$A(S,n) = \sum_{v\in S}\deg(v) = (1\pm\gamma)d \abs{S}$,
$A(n,T) = \sum_{v\in T}\deg(v) = (1\pm\gamma)d \abs{T}$,
$A(n,n) = \sum_{v\in V(G)}\deg(v) = (1\pm\gamma)d \abs{V(G)} = (1\pm\gamma)dn$.
Substitute all this into \eqref{AST simplified} and use triangle inequality to complete the proof.
\end{proof}

Sometimes all we need is at least one edge between disjoint sets of vertices $S$ and $T$. 
Corollary~\ref{cor:EML-almost-regular} provides a convenient sufficient condition for this:

\begin{cor}[At least one edge]\label{cor:expander}
Let $G$ be an $(n,(1\pm\gamma)d,\lambda)$-graph. Let $S,T\subseteq V(G)$ be two disjoint subsets with 
\[
\sqrt{\abs{S}\abs{T}} > \frac{(1+\gamma)^2}{(1-\gamma)^3} \cdot \frac{\lambda n}{d}.
\]
Then $e(S,T)>0$.
\end{cor}

\begin{proof}
Under our assumptions, the lower bound in \eqref{eq: EML-almost-regular} is strictly positive.
\end{proof}

The following statement, which is another simple corollary of the expander mixing lemma, allows us to translate minimum degree conditions into an expansion property for small sets.

\begin{lem}\label{lem:expansion}
Let $\gamma\in[0,1/20]$ be a constant, and let $\lambda\leq d/700$. Let $G$ be an $(n,(1\pm\gamma)d,\lambda)$-graph which contains subsets $S,T\subset V(G)$ such that for every $v\in S$,
$d(v,T)\geq d/6$.
Then, every subset $X\subset S$ of size $|X|\le \frac{4\lambda n}{d}$ satisfies $|N(X, T)|\geq \frac{d}{700\lambda}|X|$.
\end{lem}

\begin{proof}Let $D=\frac{d}{700\lambda}\geq 1$. Suppose that there exists a subset $X\subset S$ of size $1\le |X|\le \frac{4\lambda n}{d}$ such that $|N(X,T)|<D|X|$. Let $Y=N(X,T)$.
Corollary \ref{cor:EML-almost-regular} implies that
\[
\begin{aligned}
\frac{d|X|}{6}\leq e(X,Y)
&\leq \frac{(1+\gamma)^2d|X||Y|}{(1-\gamma)n}+\frac{1+\gamma}{1-\gamma}\cdot\lambda \sqrt{|X||Y|}\\
&\le 5\lambda D|X|+2\lambda \sqrt{D}|X|\\
&\leq 7\lambda D|X|\\
&=\frac{d|X|}{100},
\end{aligned}\]
which is a  contradiction. Therefore, every subset $X\subset S$ of size $|X|\le \frac{4\lambda n}{d}$ satisfies $|N(X,T)|\geq \frac{d}{700\lambda}|X|$. The proof is completed.
\end{proof}

\subsection{Matchings in almost regular expanders}\label{s: matchings}

In this section, we use the expander mixing lemma to get some corollaries for matchings in almost $(n,d,\lambda)$-graphs. For our convenience, we define the bipartite spectral expanders as below.

\begin{defi}\label{def:bipartite-expander}
We say that a bipartite graph $H = (V_1 \cup V_2, E)$ is an \emph{$(n, (1 \pm \gamma)d, \lambda)$-bipartite expander} if $H$ is an induced bipartite subgraph of an $(n, (1 \pm \gamma)d, \lambda)$-graph $G$ with $V(G) = V_1 \cup V_2$, and for each $i = 1, 2$ and every $v \in V_i$, the degree of $v$ in $H$ satisfies $\deg_H(v) = (1 \pm \gamma) \frac{d |V_{3-i}|}{n}$.
\end{defi}

First, we prove the existence of perfect matchings in a bipartite spectral expander with a balanced bipartition. 

\begin{lem}\label{perfect-matching-between-random-subsets}
 Let $\gamma \in[0,1/6]$ be a constant, let $d>0$ and let $\lambda\leq d/200$. 
Let $G=(V,E)$ be an $(n,(1\pm \gamma)d,\lambda)$-bipartite expander with parts $V=V_1\cup V_2$ such that $|V_1|=|V_2|$. Then $G$ contains a perfect matching.
\end{lem}

\begin{proof} 

It is enough to verify the following condition which is equivalent to Hall's condition (see Theorem 3.1.11 in \cite{west2001introduction}): for all $i\in [2]$ and $S\subseteq V_i$ of size $|S|\leq |V_i|/2$, we have $|N(S)|\geq |S|$. 

Suppose to the contrary that there exists $i\in[2]$ and an $S\subseteq V_i$, such that the set $T\coloneqq N(S)$ is of size less than $|S|$. Since $G$ is an $(n,(1\pm \gamma)d,\lambda)$-bipartite expander and since $|V_1|=|V_2|=n/2$, we have that
\[
e(S,T)\geq \frac{(1-\gamma)d}{2}\cdot |S|.
\]

On the other hand, using  the assumption that $\gamma\leq1/6$ and the expander mixing lemma for almost regular expanders (Corollary \ref{cor:EML-almost-regular}), we obtain that 
 \[
 e(S,T)\leq \frac{(1+\gamma)^2d|S||T|}{(1-\gamma)n}+\frac{1+\gamma}{1-\gamma}\cdot\lambda\sqrt{|S||T|}\leq \frac{49d|S|}{120}+\frac{7d|S|}{1000}<\frac{(1-\gamma)d}{2}\cdot |S|,
 \]
 where we also used $|T|<|S|\le n/4$ and $\lambda \le d/200$.

Combining these two estimates we obtain a contradiction. This completes the proof.
\end{proof}

If finding a perfect matching is not necessary, then following from Corollary \ref{cor:expander}, we can use a greedy algorithm to find a matching that avoids a not-too-large subset in each part of a bipartite expander.

% \blue{DM: Also need to modify the lemma because we are going to remove a part in $V_1$ as well.}

\begin{lem}\label{lem:small-matching}
Let $G$ be an $(n,(1\pm\gamma)d,\lambda)$-graph, and let $V(G)=V_1\cup V_2$ be a partition. For each $i=1,2$, let $S_i\subseteq V_i$ be a subset of size $0\leq k_i\leq |V_i|-\frac{(1+\gamma)^2}{(1-\gamma)^3} \cdot \frac{\lambda n}{d}$. Then there exists a matching of size 
\[
\min\left\{|V_1|-k_1-\frac{(1+\gamma)^2}{(1-\gamma)^3} \cdot \frac{\lambda n}{d},|V_2|-k_2-\frac{(1+\gamma)^2}{(1-\gamma)^3} \cdot \frac{\lambda n}{d}\right\}
\]
in $G$ between $V_1\setminus S_1$ and $V_2\setminus S_2$.
\end{lem}

\begin{proof}
We find the matching between $V_1\setminus S_1$ and $V_2\setminus S_2$ greedily. Initially, let $M\coloneqq \emptyset$, and let $U_1\coloneqq V_1\setminus S_1$ and $U_2\coloneqq V_2\setminus S_2$. If $|U_1|\leq \frac{(1+\gamma)^2}{(1-\gamma)^3} \cdot \frac{\lambda n}{d}$ or $|U_2|\leq \frac{(1+\gamma)^2}{(1-\gamma)^3} \cdot \frac{\lambda n}{d}$, then we stop. Otherwise, by Corollary \ref{cor:expander}, there is an edge $e\in E(G)$ between $U_1$ and $U_2$. Let $M\coloneqq M\cup\{e\}$, and let $U_1\coloneqq U_1\setminus V(e)$ and $U_2\coloneqq U_2\setminus V(e)$. Continuing in this fashion, we obtain a matching $M$ of size $\min\left\{|V_1|-k_1-\frac{(1+\gamma)^2}{(1-\gamma)^3} \cdot \frac{\lambda n}{d},|V_2|-k_2-\frac{(1+\gamma)^2}{(1-\gamma)^3} \cdot \frac{\lambda n}{d}\right\}$ in $G$ between $V_1\setminus S_1$ and $V_2\setminus S_2$. The proof is completed.
\end{proof}

\section{Extendability}\label{s: extendability}

In \cite{hyde2023spanning}, the classical tree embedding technique, introduced by Friedman and Pippenger in \cite{friedman1987expanding}, was used to prove a useful result on connecting designated pairs of vertices in expander graphs. The result shows that given two ``nice'' disjoint small subsets of an $(n,d,\lambda)$-graph, one can find a small subgraph, disjoint from these subsets, such that for any designated ordering of vertex pairs from the subsets, there exists a path factor in the subgraph that connects the pairs.
%\red{AF: why are you defining I(S) here and not before the lemma that you actually use it?}

\begin{defi}\label{def:dmextendable}Let $D,m\in\mathbb N$ with $D\ge 3$.
Let $G$ be a graph and let $S\subset G$ be a subgraph with $\Delta(S) \leq D$.
We say that $S$ is \emph{$(D,m)$-extendable}
if for all $U\subset V(G)$ with $1\le |U|\le 2m$ we have
    \[
        |(N_G(U)\cup U)\setminus V(S)|\ge (D-1)|U|-\sum_{u\in U\cap V(S)}(d_S(u)-1).
    \]
\end{defi}
The following result says that it is enough to control the external neighbourhood of small sets in order to verify extendability.

\begin{prop}\label{prop:weakerextendability}\cite{hyde2023spanning} Let $D,m\in\mathbb N$ with $D\ge 3$.
Let $G$ be a graph and let $S\subset G$ be a subgraph with $\Delta(S)\le D$.
If for all $U\subset V(G)$ with $1\leq|U|\leq 2m$ we have
\[|N_G(U)\setminus V(S)|\geq D|U|,\]
then $S$ is $(D,m)$-extendable in $G$.
\end{prop}

Before stating the lemma, we introduce some necessary definitions. We denote by $I(S)$ the edgeless subgraph with vertex set $S$. A graph $G$ is said to be \emph{$m$-joined} if for any disjoint sets $A, B \subseteq V(G)$ with $|A|, |B| \geq m$, there is at least one edge between $A$ and $B$, i.e., $e(A, B) \geq 1$.

Now we are ready to state the result from \cite{hyde2023spanning}.

\begin{lem}\label{lemma:find_sorting_network_in_expander} There is an absolute constant $C>0$ with the following property. 
Let $n$ be a sufficiently large integer, and let $20C\leq K\leq n/\log^3n$. Let $D,m\in \mathbb{N}$ satisfy $m\leq n/100D$ and $D\ge 100$. Let $G$ be an $m$-joined graph on $n$ vertices which contains disjoint subsets $V_1, V_2\subseteq V(G)$ with $|V_1|=|V_2|\leq n/K\log^{3}n$, and set $\ell:=\lfloor C \log^3 n \rfloor$.
Suppose that $I(V_1\cup V_2)$ is $(D,m)$-extendable in $G$.
\par Then, there exists a $(D,m)$-extendable subgraph $S_{res}\subseteq G$ such that for any bijection $\phi\colon V_1\to V_2$, there exists a $P_\ell$-factor of $S_{res}$ where each copy of $P_\ell$ has as its endpoints some $v\in V_1$ and $\phi(v)\in V_2$.
\end{lem}

\section{Random subgraphs of almost regular expanders}\label{s: random subgraphs}
%==============

In this section, we show that a random induced subgraph of an almost $(n,d,\lambda)$-graph or a bipartite spectral expander is typically a spectral expander by itself. This serves as our main tool in the proof of our main result.  

\subsection{Chernoff's bounds}

We extensively use the following well-known Chernoff's bounds for the upper and lower tails of the hypergeometric distribution throughout the paper. The following lemma was proved by Hoeffding \cite{hoeffding1994probability} (also see Section 23.5 in \cite{frieze2016introduction}).

\begin{lem}[Chernoff's inequality for hypergeometric distribution]\label{lem:hypergeometric}
Let 
$X\sim \mathrm{Hypergeometric}(N,K,n)$
 and let
$\mathbb{E}[X]=\mu$. Then\begin{itemize}
    \item $\Pr{X<(1-a)\mu}<e^{-a^2\mu/2}$ for every $a>0$;
    \item $\Pr{X>(1+a)\mu}<e^{-a^2\mu/3}$ for every $a\in(0,\frac{3}{2})$.
\end{itemize}
\end{lem}

\subsection{Random induced subgraphs}

The following theorem is the main result of this section. It asserts that with probability at least $1-n^{-\Theta(1)}$, random (induced) subgraphs of spectral expanders are also spectral expanders. 

\begin{thm}[Random subgraphs of spectral expanders] \label{thm:subgraphs-of-pseudorandom-graphs}
Let $\gamma\in (0,1/200]$ be a constant. There exists an absolute constant $C$ such that the following holds for sufficiently large $n$. 
    Let $d,\lambda>0$, let $\sigma\in[1/n,1)$, and let $G$ be an $(n,(1\pm\gamma)d,\lambda)$-graph. 
    Let $X\subseteq V(G)$ with $|X|=\sigma n$ be a subset chosen uniformly at random, and let $H \coloneqq G[X]$ be the subgraph of $G$ induced by $X$. 
    Assume that
    \[
    \sigma d\ge C \gamma^{-2} \log n 
    \quad \text{and} \quad
    \sigma\lambda \ge C\sqrt{\sigma d\log n}.
    \]
    Then with probability at least $1-n^{-1/6}$, $H$ is a $\left(\sigma n, (1\pm2\gamma )\sigma d, 6\sigma\lambda \right)$-graph.
\end{thm}

Let us briefly discuss the two conditions in Theorem~\ref{thm:subgraphs-of-pseudorandom-graphs}. The first condition permits the random subgraph to be quite sparse—with degrees on the order of $\log n$—but not sparser than that. Below this threshold, the degrees of the random subgraph become unstable, and it will no longer be approximately regular. The second condition is essentially the Alon-Boppana bound, up to a logarithmic factor, which dictates that the second singular value of an approximately $\sigma d$-regular graph must be at least $\Omega(\sqrt{\sigma d})$. In other words, the conditions in Theorem~\ref{thm:subgraphs-of-pseudorandom-graphs} are nearly necessary for a subgraph $H$ to be an almost regular expander. Additionally, we did not optimize the constant factor in $s_2(H)$, though we believe it should be $1 + o(1)$.

\bigskip

The proof of Theorem~\ref{thm:subgraphs-of-pseudorandom-graphs} is based on bounds of the spectral norm of a \emph{random submatrix}, which is obtained from a given $n \times n$ matrix $B$ by choosing a uniformly random subset of rows and a uniformly random subset of columns of $B$. 

There are two natural ways to choose a random subset of the set $[n]$. We can make a random subset $I$ by selecting every element of $[n]$ independently at random with probability $\sigma \in (0,1)$. In this case, we write 
\[
I \sim \SSubset(n,\sigma).
\]
Alternatively, we can choose any $m$-set $J$ of $[n]$ with the same probability $1/\binom{n}{m}$. In this case, we write 
\[
J \sim \SSubset(n,m).
\]

Note that if $m = \sigma n$, the models $\SSubset(n,\sigma)$ and $\SSubset(n,m)$ are closely related but not identical. It should be clear from the context which one we consider. 

For a given subset $I \subset [n]$, we denote by $P_I$ the orthogonal projection in $\R^n$ onto $\R^I$. In other words, $P_I$ is the diagonal matrix with $P_{ii}=1$ if $i \in I$ and $P_{ii}=0$ if $i \not\in I$.

The main tool of this section is the following bound. It is worth mentioning that several similar results have been proved before, for example, in \cite{rudelson2007sampling} and \cite{tropp2008norms}.

\begin{thm}[Norms of random submatrices] \label{thm: random submatrices}
    Let $B$ be an $n\times n$ matrix. Let $I,I' \sim \SSubset(n,\sigma)$ be two independent subsets, where $\sigma\in(0,1)$. Let $p\geq 2$ and let $q=\max\{p,2\log n\}$. Then
    \[
    \mathbb{E}_p\norm{P_I B P_{I'}} 
    \le \sigma\norm{B} + 3\sqrt{q\sigma}\left( \norm{B}_{1\rightarrow 2} + \norm[1]{B^\tran}_{1\rightarrow 2} \right) + 8q\norm{B}_\infty.
    \]
    Here $\mathbb{E}_p[X]=(\mathbb{E}|X|^p)^{1/p}$ is the $L_p$ norm of the random variable $X$; the norm $\|{\mkern 2mu\cdot\mkern 2mu}\|_{1\to2}$ denotes the norm of a matrix as an $\ell_1 \to \ell_2$ linear operator, which  equals to the maximum value among the $\ell_2$ norm of each column; 
    and $\norm{\cdot}_\infty$ denotes the maximum absolute entry of a matrix.
\end{thm}

We use the following results to derive Theorem \ref{thm: random submatrices}.

\begin{lem}[Rudelson-Vershynin \cite{rudelson2007sampling}]       
    \label{lem:RV}
    Let $A$ be an $m\times n$ matrix with rank $r$.
    Let $I \sim \SSubset(n,\sigma)$, where $\sigma \in (0,1)$.
    Let $p \ge 2$ and let $q=\max\{p,2\log r\}$. Then
    \[\mathbb{E}_p\norm{AP_I} \leq \sqrt{\sigma}\norm{A} + 3\sqrt{q}\mathbb{E}_p\norm{AP_I}_{1\rightarrow 2}.\] 
\end{lem}

\begin{thm}[Tropp~\cite{tropp2008norms}]\label{thm:tropp}
Let $A$ be an $m\times n$ matrix with rank $r$.
    Let $I \sim \SSubset(n,\sigma)$, where $\sigma \in (0,1)$.
    Let $p \ge 2$ and let $q=\max\{p,2\log r\}$. Then
    \[
    \mathbb{E}_p\|P_I A\|_{1\rightarrow 2}
    \le \sqrt{\sigma}\norm{A}_{1\rightarrow 2} + 2^{1.25}\sqrt{q}\mathbb{E}_p\norm{P_IA}_\infty
    \]
\end{thm}

\begin{proof}[Proof of Theorem~\ref{thm: random submatrices}]

 First, we apply Lemma~\ref{lem:RV} twice (in the same manner as in \cite{tropp2008norms}), where we first take $P_IB,P_{I'}$ in place of $A,P_I$, and then take $B^\tran,P_I$ in place of $A,P_I$. So we obtain
    \[
    \begin{aligned}
        \mathbb{E}_p\norm{P_I B P_{I'}}
        &\le \sqrt{\sigma}\mathbb{E}_p\norm{P_I B} + 3\sqrt{q}\mathbb{E}_p\norm{P_I B P_{I'}}_{1\rightarrow 2}\\
        &\le \sigma\norm{B} + 3\sqrt{q\sigma}\mathbb{E}_p\norm[1]{B^\tran P_I}_{1\rightarrow 2} + 3\sqrt{q}\mathbb{E}_p\norm{P_I B P_{I'}}_{1\rightarrow 2} \\
        &\le \sigma\norm{B} + 3\sqrt{q\sigma}\norm[1]{B^\tran}_{1\rightarrow 2} + 3\sqrt{q}\mathbb{E}_p\norm{P_I B}_{1\rightarrow 2},
    \end{aligned}
    \]
    where the last inequality follows since the $1 \to 2$ norm of a submatrix is bounded by the $1 \to 2$ norm of a matrix. We then use Theorem \ref{thm:tropp} to complete the proof: 
    \pushQED{\qed}\begin{align*}\mathbb{E}_p\|P_I B\|_{1\rightarrow 2}
    &\le \sqrt{\sigma}\norm{B}_{1\rightarrow 2} + 2^{1.25}\sqrt{q}\mathbb{E}_p\norm{P_IB}_\infty\\
    &\leq
    \sqrt{\sigma}\norm{B}_{1\rightarrow 2} + 2^{1.25}\sqrt{q}\norm{B}_\infty\\
    &\leq \sqrt{\sigma}\norm{B}_{1\rightarrow 2} + 8\sqrt{q}\norm{B}_\infty/3.
    \qedhere
    \end{align*}

\end{proof}

Theorem \ref{thm: random submatrices} gives a tool when one wishes to study the case of independent random subsets of row indices and column indices. However, because our goal is to study random subsets of fixed size of a given set, we cannot apply Theorem \ref{thm: random submatrices} directly since now the selections of row indices and column indices are not independent. Instead, we would like to make $I=I'$ and change the model of sampling. The following tools make this possible.

\begin{lem}
    [Decoupling \cite{tropp2008norms}] \label{decoupling}
    Let $B$ be a diagonal-free symmetric $n \times n$ matrix. 
    Let $I,I' \sim \SSubset(n,\sigma)$ be two independent subsets, where $\sigma \in (0,1)$.
    Then for every $p \ge 2$, we have
    \[\mathbb{E}_p\norm{P_I B P_I}\leq 2\mathbb{E}_p\norm{P_I B P_{I'}}.\] 
\end{lem}

\begin{lem}[Random subset models \cite{tropp2008random}]    \label{random models}
    Let $B$ be an $n \times n$ matrix.
    Let $I \sim \SSubset(n,\sigma)$ and $J \sim \SSubset(n,m)$ be two independent subsets,  where $\sigma \in (0,1)$ and $m=\sigma n\ge 1$. Then
    for every $p\geq 2$, we have
    \[
    \mathbb{E}_p\norm{P_J B P_J}\leq 2^{1/p}\mathbb{E}_p\norm{P_I B P_I}.
    \]
\end{lem}

By combining the two lemmas above and Theorem \ref{thm: random submatrices}, we can obtain a corollary as follows:

\begin{cor}[Norms of random submatrices] \label{cor: random submatrix}
    Let $B$ be a symmetric $n\times n$ matrix.
    Let $J \sim \SSubset(n,m)$, where $\sigma \in (0,1)$ and $m=\sigma n\ge 1$.
    Let $p\geq 2$ and let $q=\max\{p,2\log n\}$. Then
    \[
    \mathbb{E}_p\norm{P_J B P_J}
    \le 4\sigma\norm{B} + 24\sqrt{q\sigma}\norm{B}_{1\rightarrow 2} + 35q\norm{B}_\infty.
    \]
\end{cor}

\begin{proof}
    Consider the symmetric, diagonal-free matrix $B_0=B-D$ where $D \coloneqq \diag(B_{1,1},\ldots,B_{n,n})$. 
    Combining Theorem~\ref{thm: random submatrices} with Lemmas~\ref{decoupling} and \ref{random models}, we obtain the following:
    \[
    \mathbb{E}_p\norm{P_J B_0 P_J} 
    \le 4\sigma\norm{B_0} + 24\sqrt{q\sigma}\norm{B_0}_{1\rightarrow 2} + 32q\norm{B_0}_\infty.
    \]
    Note that $\norm{B_0} \le \norm{B}+\norm{D}$, $\norm{B_0}_{1\rightarrow 2} \le \norm{B}_{1\rightarrow 2}$, $\norm{B_0}_\infty \le \norm{B}_\infty$, and 
    $\norm{P_JBP_J} \le \norm{P_JB_0P_J}+\norm{P_JDP_J}\le \norm{P_JB_0P_J}+\norm{D}$. This implies
    \[
    \mathbb{E}_p\norm{P_J B P_J} \leq \mathbb{E}_p\norm{P_JB_0P_J} + \norm{D}
    \le 4\sigma \left(\norm{B}+\norm{D}\right) + 24\sqrt{q\sigma}\norm{B}_{1\rightarrow 2} + 32q\norm{B}_\infty + \norm{D}.
    \]
    Notice that $\norm{D}=\max_i\abs{B_{i,i}} \le \norm{B}_\infty$ to complete the proof.
\end{proof}

\medskip

We are now ready to prove Theorem \ref{thm:subgraphs-of-pseudorandom-graphs}.

\begin{proof}[Proof of Theorem~\ref{thm:subgraphs-of-pseudorandom-graphs}]
    Let $C>0$ be a sufficiently large absolute constant. To see that the random induced subgraph $H\coloneqq G[X]$ is almost regular whp, we can apply Lemma \ref{lem:hypergeometric} with $n,(1\pm\gamma)d,\sigma n,\gamma/(1+\gamma)$ in place of $N,K,n,a$. Since $\sigma d\geq C\gamma^{-2}\log n$ for sufficiently large absolute constant $C>0$, it follows that, with probability at least $1-n^{-1}$, all degrees of $H$ are $(1\pm2\gamma)\sigma d$.
    Thus, it remains to bound the second singular value of $A_H$ whp, where $A_H$ is the adjacency matrix of $H$. 
    
    It is convenient to first work with normalized matrices. So let us consider the normalized adjacency matrix
    \begin{equation}    \label{eq: AGnormalized}
        \bar{A}_G=D^{-1/2}A_GD^{-1/2},
        \quad \text{where} \quad D=\diag(d_1,\ldots,d_n)    
    \end{equation}
    is the degree matrix of $G$. Note that  for any $m\times n$ matrix 
    $A$, we have that $s_2(A)=\min_B\:\norm{A-B}$, where the minimum is over all rank-one $m\times n$ matrices $B$ (see Lemma \ref{lem:low-rank-approximation}). Thus, by applying Observation~\ref{obs: |N|}, we have 
    \begin{equation}    \label{eq: B}
        s_2(\bar{A}_G)=\norm{B},
        \quad \text{where} \quad
        B = \bar{A}_G - \frac{1}{a} D^{1/2} \mathbbm{1}_n \mathbbm{1}_n^\tran D^{1/2}
        \quad \text{and } a = \sum_{i=1}^n d_i.
    \end{equation}

    Applying Corollary~\ref{cor: random submatrix} for any $p\geq 2$ and $q=\max\{p,2\log n\}$, we obtain
    
    \begin{equation}    \label{eq: E PXBPX}
        \mathbb{E}_p\|P_X B P_X\|
        \leq 4\sigma\|B\| + 24\sqrt{q\sigma}\|B\|_{1\rightarrow 2} + 35q\norm{B}_\infty.            
    \end{equation}
    Let us bound each of the three terms on the right hand side.

    \medskip

    {\em Bounding $\norm{B}$.} First, by \eqref{eq: B}, Corollary~\ref{cor: singular values normalized} and the assumptions, we have
    \begin{equation}    \label{eq: Bop bounded}
        \norm{B} 
        = s_2(\bar{A}_G) 
        \le \frac{s_2(A_G)}{(1-\gamma)d}
        \leq \frac{1.1\lambda}{d}. 
    \end{equation}
    
    \medskip

    {\em Bounding $\norm{B}_{1\rightarrow 2}$.} Triangle inequality yields
    \begin{equation}    \label{eq: B12}
        \|B\|_{1\rightarrow 2} 
        \le \norm{\bar{A}_G}_{1\rightarrow 2} + \frac{1}{a} \norm[1]{D^{1/2} \mathbbm{1}_n \mathbbm{1}_n^\tran D^{1/2}}_{1\rightarrow 2}.        
    \end{equation}
    Let us bound each of the terms appearing on the right hand side. First, 
    \[
    \norm{\bar{A}_G}_{1\rightarrow 2} 
    = \norm[1]{D^{-1/2}A_GD^{-1/2}}_{1\rightarrow 2}
    \le \norm[1]{D^{-1}} \norm{A_G}_{1\rightarrow 2}.
    \]
    We have $\norm{D^{-1}} = \max_i (1/d_i) \le 1.1/d$ 
    and $\norm{A_G}_{1\rightarrow 2} = \max_j \sqrt{d_j} \le 1.1\sqrt{d}$. Thus, \begin{equation}    \label{eq: AGbar}
        \norm{\bar{A}_G}_{1\rightarrow 2} \le \frac{1.3}{\sqrt{d}}.   
    \end{equation}
    Next, \begin{equation}    \label{eq: a}
    a = \sum_{i=1}^n d_i\ge (1-\gamma)dn \ge 0.9dn.
    \end{equation}
    Moreover, 
    \begin{equation}    \label{eq: D11D 12}
        \norm[1]{D^{1/2} \mathbbm{1}_n \mathbbm{1}_n^\tran D^{1/2}}_{1\rightarrow 2}
        \le \norm{D} \cdot \norm{\mathbbm{1}_n \mathbbm{1}_n^\tran}_{1\rightarrow 2}
        \le (1+\gamma)d \cdot \sqrt{n} \le 1.1d\sqrt{n}.   
    \end{equation}
    
    Putting \eqref{eq: AGbar}, \eqref{eq: a} and \eqref{eq: D11D 12} into \eqref{eq: B12}, we get
    \begin{equation}    \label{eq: B12 bounded}
        \|B\|_{1\rightarrow 2} 
        \le \frac{1.3}{\sqrt{d}} + \frac{1}{0.9dn} \cdot 1.1d\sqrt{n}
        \le \frac{2.6}{\sqrt{d}}.
    \end{equation}

    {\em Bounding $\norm{B}_\infty$.}
    Again, triangle inequality yields
    \begin{equation}    \label{eq: Binfty}
        \|B\|_\infty 
        \le \norm{\bar{A}_G}_\infty + \frac{1}{a} \norm[1]{D^{1/2} \mathbbm{1}_n \mathbbm{1}_n^\tran D^{1/2}}_\infty.        
    \end{equation}
    All entries of $\bar{A}_G$ are $1/\sqrt{d_id_j} \le 1.1/d$, 
    and all entries of $D^{1/2} \mathbbm{1}_n \mathbbm{1}_n^\tran D^{1/2}$ are $\sqrt{d_id_j} \le 1.1d$. Also, recall that $a \ge 0.9dn$ by \eqref{eq: a}. Thus, plugging them into \eqref{eq: Binfty}, we obtain
    \begin{equation}    \label{eq: Binfty bounded}
        \norm{B}_\infty 
        \le \frac{1.1}{d} + \frac{1}{0.9dn} \cdot 1.1d
        \le \frac{2.4}{d}.
    \end{equation}

    \medskip

    Putting \eqref{eq: Bop bounded}, \eqref{eq: B12 bounded} and \eqref{eq: Binfty bounded} into \eqref{eq: E PXBPX}, we obtain
    \[
    \mathbb{E}_p \norm{P_X B P_X}
    \le \frac{4.4\sigma\lambda}{d} + 63\sqrt{\frac{q\sigma}{d}} + \frac{84q}{d}.
    \]
    Multiplying on the left and right by $D^{1/2}$ inside the norm, we conclude that 
    \[
    \mathbb{E}_p \norm{D^{1/2} P_X B P_X D^{1/2}}
    \le \norm{D} \mathbb{E}_p \norm{P_X B P_X}
    \le 5\sigma\lambda + 70\sqrt{ q\sigma d} + 93q
    \eqqcolon \lambda_0,
    \]
    where we used that $\norm{D}=\max_i d_i \le 1.1d$.
    Since diagonal matrices commute, we can express the matrix above as follows:
    \[
    D^{1/2} P_X B P_X D^{1/2} 
    = P_X D^{1/2} B D^{1/2} P_X
    = P_X A_G P_X - \frac{1}{a} P_X D \mathbbm{1}_n \mathbbm{1}_n^\tran D P_X,
    \]
    where in the last step we used \eqref{eq: AGnormalized} and \eqref{eq: B}.
    Note that $\frac{1}{a} P_X D \mathbbm{1}_n \mathbbm{1}_n^\tran D P_X$ is a rank one matrix. Thus, by Lemma~\ref{lem:low-rank-approximation}, we have $s_2(P_X A_G P_X) \le \norm{D^{1/2} P_X B P_X D^{1/2}}$, and thus
    \[
    \E_p s_2(P_X A_G P_X) \le \lambda_0.
    \]
    Since the adjacency matrix $A_H$ of the induced subgraph $H$ is a $\sigma n\times \sigma n$ submatrix of the $n \times n$ matrix $P_X A_G P_X$, by the Interlacing Theorem for singular values (Theorem \ref{thm:Interlacing-Theorem}),  it follows that 
    \[
    \E_p s_2(A_H) \le \lambda_0.
    \]
    
    Now choose $p=2\log n$ and thus $q=p=2\log n$.
    Applying Markov's inequality, we obtain
    \[\begin{aligned}
        \Pr{s_2(A_H) \ge 1.1\lambda_0}
        &= \Pr{s_2(A_H)^p \ge (1.1\lambda_0)^p}
        \le \left( \frac{\mathbb{E}_p s_2(A_H)}{1.1\lambda_0}\right)^p \\
        &\le (1.1)^{-p}
        =(1.1)^{-2\log n} \le n^{-0.19}.
    \end{aligned}
    \]
    %\blue{DM: Here, $(1.1)^{-2\log n}=n^{-2\times\log(1.1)}\approx n^{-0.0828}$.}
    In other words, with probability at least $1- n^{-0.19}$, we have
    \[
    s_2(A_H) < 1.1\lambda_0 
    \le 5.5\sigma\lambda + 109\sqrt{ \sigma d\log n} + 205 \log n.
    \]
    
    To complete the proof, we show that the first term dominates the right hand side.
    Indeed, since the absolute constant $C$ is sufficiently large, the first condition in Theorem \ref{thm:subgraphs-of-pseudorandom-graphs} implies that $205 \log n \le \sqrt{\sigma d \log n}$. Similarly, the second condition in the theorem implies that $110\sqrt{ \sigma d\log n} \le 0.5 \sigma\lambda$. Then it follows that 
    \[
    s_2(A_H) \le 5.5\sigma\lambda + 0.5 \sigma\lambda = 6 \sigma\lambda.
    \]
    Therefore, with probability at least $(1-n^{-1})(1-n^{-0.19})\geq 1-n^{-1/6}$, $H$ is a $\left(\sigma n, (1\pm2\gamma )\sigma d, 6\sigma\lambda \right)$-graph,  which completes the proof of Theorem~\ref{thm:subgraphs-of-pseudorandom-graphs}.
\end{proof}

\medskip

Sometimes we will work on the bipartite subgraph $G[X,Y]$ induced by random disjoint subsets $X,Y\subseteq V(G)$. We also have $G[X,Y]$ is a bipartite spectral expander whp, which is a direct corollary of Theorem \ref{thm:subgraphs-of-pseudorandom-graphs}.

\begin{cor}\label{cor:subgraphs-of-pseudorandom-bipartite-graphs}
Let $\gamma\in (0,1/200]$ be a constant. There exists an absolute constant $C$ such that the following holds for sufficiently large $n$. 
    Let $d,\lambda>0$, let $\sigma_1,\sigma_2\in[1/n,1)$, and let $G$ be an $(n,(1\pm\gamma)d,\lambda)$-graph. 
    Let $X,Y\subseteq V(G)$ with $|X|=\sigma_1 n$ and $|Y|=\sigma_2n$
    be two disjoint  subsets chosen uniformly at random, and let $H \coloneqq G[X,Y]$ be the bipartite subgraph of $G$ induced by $X$ and $Y$. 
    Let $\sigma\coloneqq \sigma_1+\sigma_2$. Assume that
    \[
    \sigma_1d,\sigma_2 d\ge C \gamma^{-2} \log n 
    \quad \text{and} \quad
    \sigma\lambda \ge C\sqrt{\sigma d\log n}.
    \]
    Then with probability at least $1-n^{-1/7}$, $H$ is a $\left(\sigma n, (1\pm2\gamma )\sigma d, 6\sigma\lambda \right)$-bipartite expander.  
\end{cor}

\begin{proof}
Let $C>0$ be a sufficiently large absolute constant. Since $\sigma_1n,\sigma_2n\geq C\gamma^{-2}\log n$, by Chernoff's bounds, we have that 
\[
\Pr{\exists v\in V, \deg(v,X)\neq \left(1\pm2\gamma\right)\sigma_1n}\leq n^{-1}
\]
and
\[
\Pr{\exists v\in V, \deg(v,Y)\neq \left(1\pm2\gamma\right)\sigma_2n}\leq n^{-1}.
\]

Next, note that $X\cup Y$ is a random subset of size $|X|+|Y|=\sigma_1n+\sigma_2n=\sigma n$. Now, since $\sigma d=\sigma_1d+\sigma_2d\geq 2C\gamma^{-2}\log n$ and $\sigma\lambda\geq C\sqrt{\sigma d\log n}$, Theorem \ref{thm:subgraphs-of-pseudorandom-graphs} implies that with probability at least $1-n^{-1/6}$,
\[
G[X\cup Y]\text{ is a }\left(\sigma n,(1\pm2\gamma)\sigma d,6\sigma\lambda\right)\text{-graph}.
\]

Therefore, we have that with probability at least $(1-2n^{-1})(1-n^{-1/6})\geq 1-n^{-1/7}$,
\[
H\coloneqq G[X,Y]\text{ is a }\left(\sigma n,(1\pm2\gamma)\sigma d,6\sigma\lambda\right)\text{-bipartite expander},
\]
which completes the proof.
\end{proof}

\section{Proof of Theorem \ref{main-thm}}\label{s: main proof}

In this section, we prove our main result, Theorem \ref{main-thm}. Since regular spectral expanders can be viewed as $(n, (1 \pm \gamma)d, \lambda)$-graphs, the following slightly stronger statement will directly imply Theorem \ref{main-thm}.

\begin{thm}\label{generalized-main-thm}
Let $\gamma\in (0,1/400]$ and let $n$ be a sufficiently large integer. Then,  
any $(n,(1\pm \gamma)d,\lambda)$-graph 
with $\lambda\leq d/70000$ and $d \geq \log^{6}n$
contains a Hamilton cycle.
\end{thm}

Before proving the theorem, we first state and prove the following simple averaging argument.

\begin{claim}\label{claim:double-counting-same-set}
Let $\alpha\in[0,1]$, let $0<m\leq h\leq n/2$ be integers, and let $\mathcal{P}$ be any graph property. Let $G=(V,E)$ be a graph on $n$ vertices. Suppose that there are at least $1-\alpha$ proportion of pairs of disjoint $m$-sets $A,B\subseteq V$ such that $G[A,B]\in\mathcal{P}$. Let $\mathcal{F}\subseteq \binom Vm$ be the family of $m$-sets $A$ such that for at least $1-\alpha^{1/2}$ proportion of $m$-sets $B\subseteq V\setminus A$, $G[A,B]\in\mathcal{P}$. For $A\in \mathcal F$, let $\mathcal F_A\subseteq \binom{V\setminus A}m$ be the family of $m$-sets $B\subseteq V\setminus A$ such that $G[A,B]\in\mathcal{P}$. Then the following properties hold:
\begin{enumerate}[label=(A\arabic*)]
    \item\label{A1} $|\mathcal{F}|\ge (1-\alpha^{1/2})\binom{n}{m}$;
    \item\label{A2} for a uniformly random $h$-set $S\subseteq V\setminus A$, with probability at least $1-\alpha^{1/4}$, $|\mathcal F_A\cap \binom Sm|\ge (1-\alpha^{1/4})\binom hm$;
    \item\label{A3} for $A\in\binom Vm$ and $S\in \binom Vh$ chosen uniformly at random such that $A\cap S=\emptyset$, with probability at least $1-\alpha^{1/2}-\alpha^{1/4}$, for at least $1-\alpha^{1/4}$ proportion of $m$-sets $B\subseteq S$, $G[A,B]\in\mathcal{P}$.
\end{enumerate}
\end{claim}

\begin{proof} 
We prove the statements in sequence. 

\medskip

{\em Property \ref{A1}.} Suppose to the contrary that there are at least $\alpha^{1/2}\binom{n}{m}$ $m$-sets $A\notin \mathcal{F}$. 
By definition, each such $A$ contributes at least $\alpha^{1/2}\binom{n-m}{m}$ pairs $(A,B)$ such that $G[A,B]\notin\mathcal{P}$. 
Therefore, there are at least $\alpha^{1/2}\binom {n}{m}\cdot \alpha^{1/2}\binom{n-m}{m} = \alpha\binom{n}{m}\binom{n-m}{m}$ pairs $(A,B)$ such that $G[A,B]\notin\mathcal{P}$, contradicting the assumption that there are at least $1-\alpha$ proportion of pairs $(A,B)$ with $G[A,B]\in\mathcal{P}$. Thus, $|\mathcal{F}|\ge (1-\alpha^{1/2})\binom{n}{m}$.

\medskip

{\em Property \ref{A2}.} Suppose to the contrary that there are more than $\alpha^{1/4}\binom{n-m}{h}$ $h$-sets $S\subseteq V\setminus A$, each containing at least $\alpha^{1/4}\binom{h}{m}$ $m$-sets $B$ such that $B\notin \mathcal{F}_A$. Since each such subset $B$ is counted at most $\binom{n-2m}{h-m}$ times, there are in total more than 
\[
\frac{1}{\binom{n-2m}{h-m}}\cdot \alpha^{1/4}\binom{n-m}{h}\cdot \alpha^{1/4}\binom{h}{m} = \alpha^{1/2}\binom{n-m}{m}
\]
$m$-sets $B$ such that $B\notin\mathcal{F}_A$. This contradicts the assumption that $|\mathcal F_A|\ge (1-\alpha^{1/2})\binom{n-m}{m}$.

\medskip

{\em Property \ref{A3}.} By applying \ref{A1}, \ref{A2}, and the union bound, for disjoint uniformly randomly chosen $A\in \binom Vm$ and $S\in \binom Vh$, with probability at least $1-\alpha^{1/2}-\alpha^{1/4}$, we have that $A\in \mathcal F$ and $|\mathcal F_A\cap \binom Sm|\ge (1-\alpha^{1/4})\binom hm$.
This implies that for at least $1-\alpha^{1/4}$ proportion of $m$-sets $B\subseteq S$, $G[A,B]\in\mathcal{P}$. 

\medskip

The proof is completed.
\end{proof}

We are now ready to prove Theorem \ref{generalized-main-thm}.

\begin{proof}[Proof of Theorem \ref{generalized-main-thm}]

Assume that $\gamma\leq 1/400$, $\lambda\leq d/70000$, 
and that $G=(V,E)$ is an $(n,(1\pm\gamma)d,\lambda)$-graph with a sufficiently large integer $n$ and $d\geq \log^6n$. Since an $(n,(1\pm\gamma)d,\lambda)$-graph is also an $(n,(1\pm\gamma)d,\lambda')$-graph when $\lambda\leq\lambda'$, we may assume that $\sqrt{d}\log^3n\leq \lambda \leq d/70000 $. Let  $k\coloneqq n/\log^4n$.

\subsection{Partitioning the graph}

First, we find a partition of the vertex set of the $(n,(1\pm \gamma)d,\lambda)$-graph $G$ with some nice properties.

\begin{claim}\label{claim:obatin-endpoints}
There exists a partition $V=X_1\cup X_2\cup Y_1\cup Y_2\cup R_1\cup R_2$ with $|X_1|=|Y_1|=\frac{k}{5}$, $|X_2|=|Y_2|=\frac{4k}{5}$ and $|R_2|=\frac{4n}{5}$, such that the following properties hold: 
\begin{enumerate}[label=(P\arabic*)]
    \item\label{P1} for every vertex $v\in V$, we have: \[
    \begin{aligned}
    \deg(v,X_1),\deg(v,Y_1)&= (1\pm2\gamma)\frac{dk}{5n},\quad \deg(v,X_2),\deg(v,Y_2)= (1\pm2\gamma)\frac{4dk}{5n}\\    
    \deg(v,R_1)&=(1\pm2\gamma)\frac{d}{5},\quad\text{and}\quad \deg(v,R_2)=(1\pm2\gamma)\frac{4d}{5};
    \end{aligned}
    \]

    \item \label{P2}
    letting $X\coloneqq X_1\cup X_2$ and $Y\coloneqq Y_1\cup Y_2$, the subgraph
    \[
    G'\coloneqq G[X\cup Y\cup R_1]\text{ is an }\left(\frac{n}{5},(1\pm2\gamma)\frac{d}{5}, \frac{6\lambda}{5}\right)\text{-graph};
    \]

    \item \label{P3}
    for at least $1-n^{-1/7}$ proportion of disjoint subsets $S,T\subseteq R_1\cup R_2$ of equal size $k$,
    \[
    G[S,T]\text{ is a }\left(2k,(1\pm4\gamma)\frac{2dk}{n},\frac{18\lambda k}{n}\right)\text{-bipartite expander};
    \]

    \item \label{P4}
    for at least $1-n^{-1/28}$ proportion of $k$-sets $S\subseteq R_1\cup R_2$, the bipartite subgraphs 
    \[G[X,S]\text{ and }G[S,Y]\text{ are }\left(2k,(1\pm2\gamma)\frac{2dk}{n},\frac{12\lambda k}{n}\right)\text{-bipartite expanders};
    \]

    \item \label{P5}

    the bipartite subgraph
    \[
    G[X,Y]\text{ is a }\left(2k,(1\pm2\gamma)\frac{2dk}{n}, \frac{12\lambda k}{n}\right)\text{-bipartite expander}.
    \]

\end{enumerate}
\end{claim}

\begin{proof}[Proof of Claim~\ref{claim:obatin-endpoints}]
Let $V=X_1\cup X_2\cup Y_1\cup Y_2\cup R_1\cup R_2$ be a uniformly random partition  with $|X_1|=|Y_1|=\frac{k}{5}$, $|X_2|=|Y_2|=\frac{4k}{5}$ and $|R_2|=\frac{4n}{5}$. We wish to prove that each property among \ref{P1}--
\ref{P5} holds whp.

\medskip

{\em Property~\ref{P1}.} Since for all $v\in V$ we have $\mathbb{E}[\deg(v,X_1)]=(1\pm\gamma)\frac{dk}{5n}=\omega(\log n)$, it follows by Chernoff's bounds and the union bound that 
\[
\Pr{\exists v\in V, \deg(v,X_1)\neq \left(1\pm2\gamma\right)\frac{dk}{5n}}\leq ne^{-\omega(\log n)}=o(1).
\]
% and
% \[
% \Pr{\exists v\in V, \deg(v,X)\neq(1\pm2\gamma)\frac{dk}{n}\text{ or }\deg(v,Y)\neq(1\pm2\gamma)\frac{dk}{n}}\leq ne^{-\omega(\log n)}=o(1).
% \]
Similarly, the other degree bounds also hold with probability at least $1 - o(1)$. Therefore, property \ref{P1} holds with high probability. For the remainder of the proof, we will condition on property \ref{P1}.

\medskip

{\em Property~\ref{P2}.} Note that $X\cup Y\cup R_1$ is a uniformly random subset of size $\frac{n}{5}$. Since  $\frac{d}{5}=\omega(\gamma^{-2}\log n)$ and  $\frac{\lambda}{5}=\omega\left(\sqrt{\frac{d\log n}{5}}\right)$, 
by Theorem \ref{thm:subgraphs-of-pseudorandom-graphs}, we have that with probability at least $1-n^{-1/6}$, \[
    G'= G[X\cup Y\cup R_1]\text{ is an }\left(\frac{n}{5},(1\pm2\gamma)\frac{d}{5}, \frac{6\lambda}{5}\right)\text{-graph}.
    \]

\medskip

{\em Property~\ref{P3}.}
First, we prove that $G[R_1\cup R_2]$ is an $(n-2k,(1\pm\frac{3}{2}\gamma)d,\lambda)$-graph. Indeed, by property \ref{P1}, for every vertex $v\in V$, 
\[
\begin{aligned}
\deg(v,R_1\cup R_2)&=
\deg(v)-\deg(v,X_1)-\deg(v,X_2)-\deg(v,Y_1)-\deg(v,Y_2)\\
&=(1\pm\gamma)d-2(1\pm2\gamma)\left(\frac{dk}{5n}+\frac{4dk}{5n}\right)\\
&=\left(1\pm\frac{3}{2}\gamma\right)d.
\end{aligned}
\]
Also, by the Interlacing Theorem for singular values (Theorem \ref{thm:Interlacing-Theorem}), we have that $s_2(G[R_1\cup R_2])\leq s_2(G)\leq \lambda$. Therefore, $G[R_1\cup R_2]$ is an $(n-2k,(1\pm\frac{3}{2}\gamma)d,\lambda)$-graph.

Next, let $S,T\subseteq R_1\cup R_2$ be two disjoint $k$-sets chosen uniformly at random.  Since $(1+\frac{1}{2}\gamma)\frac{2dk}{n}\geq\frac{2dk}{n-2k}=\omega((2\gamma)^{-2}\log n)$ and  $
\frac{3\lambda k}{n}\geq \frac{2\lambda k}{n-2k}=\omega\left(\sqrt{\frac{2dk\log n}{n-2k}}\right)$,  Theorem \ref{thm:subgraphs-of-pseudorandom-graphs} implies that with probability at least $1-n^{-1/6}$, 
\[
G[S\cup T]\text{ is a }\left(2k,(1\pm4\gamma)\frac{2dk}{n},\frac{18\lambda k}{n}\right)\text{-graph}.
\]
Also, by Chernoff's bounds, with probability at least (say) $1-n^{-1}$, for every vertex $v\in V$, \[
\deg(v,S),\deg(v,T)=(1\pm4\gamma
)\frac{dk}{n}.
\]
Thus, by the union bound, for at least $(1- n^{-1/6})(1- n^{-1})\geq 1- n^{-1/7}$ proportion of disjoint subsets $S,T\subseteq R_1\cup R_2$ of equal size $k$,
    \[
    G[S,T]\text{ is a }\left(2k,(1\pm4\gamma)\frac{2dk}{n},\frac{18\lambda k}{n}\right)\text{-bipartite expander.}
    \]

\medskip

{\em Property~\ref{P4}.}  We only prove the property for $X$, while the analogous  statement for $Y$ can be shown similarly. By applying Corollary \ref{cor:subgraphs-of-pseudorandom-bipartite-graphs}, we can obtain that for at least $1-n^{-1/7}$ proportion of disjoint subsets $A,B\subseteq V$ of equal size $k$,
\[
G[A,B]\text{ is a }\left(2k,(1\pm2\gamma)\frac{2dk}{n},\frac{12\lambda k}{n}\right)\text{-bipartite expander}.
\]  

Note that 
$X$ is a uniformly random $k$-set and  $R_1\cup R_2\subseteq V\setminus X$ is a uniformly random $(n-2k)$-set. Thus, by \ref{A3} applied with $X,R_1\cup R_2,k,n-2k,n^{-1/7}$ in place of $A,S,m,h,\alpha$, we have that with probability at least $1-n^{-1/14}-n^{-1/28}\geq 1-n^{-1/29}$, for at least $1-n^{-1/28}$ proportion of $k$-sets $S\subseteq R_1\cup R_2$, 
\[
G[X,S]\text{ is a }\left(2k,(1\pm2\gamma)\frac{2dk}{n},\frac{12\lambda k}{n}\right)\text{-bipartite expander}.
\]

\medskip

{\em Property~\ref{P5}.} Note that $X\cup Y$ is a uniformly random subset of size $2k$. Since  $\frac{2dk}{n}=\omega(\gamma^{-2}\log n)$ and  $\frac{2\lambda k}{n}=\omega\left(\sqrt{\frac{2dk\log n}{n}}\right)$,  
by Theorem \ref{thm:subgraphs-of-pseudorandom-graphs}, we have that with probability at least $1-n^{-1/6}$, \[
    G[X\cup Y]\text{ is a }\left(2k,(1\pm2\gamma)\frac{2dk}{n}, \frac{12\lambda k}{n}\right)\text{-graph}.
    \]
Now by \ref{Q1}, the bipartite subgraph 
\[
    G[X,Y]\text{ is a }\left(2k,(1\pm2\gamma)\frac{2dk}{n}, \frac{12\lambda k}{n}\right)\text{-bipartite expander}.
    \]

\medskip

All in all, with positive probability all properties \ref{P1}--\ref{P5} hold, which guarantees a partition $V=X_1\cup X_2\cup Y_1\cup Y_2\cup R_1\cup R_2$ satisfying all desired properties. This completes the proof.
\end{proof}

\subsection{Finding $S_{res}$ and partitioning $(R_1\cup R_2)\setminus V(S_{res})$}

We pick a partition $V=X_1\cup X_2\cup Y_1\cup Y_2\cup R_1\cup R_2$ as in Claim \ref{claim:obatin-endpoints}. In order to find the subgraph $S_{res}$ which will be used to close a collection of vertex-disjoint paths into a cycle, we first verify the assumptions of Lemma \ref{lemma:find_sorting_network_in_expander}. Recall that $G'=G[X\cup Y\cup R_1]$.  Let $D\coloneqq \frac{d}{700\lambda}\geq 100$, and let $m\coloneqq \frac{(1+2\gamma)^2}{(1-2\gamma)^3}\cdot \frac{(6\lambda/5)\cdot(n/5)}{d/5}+1\leq \frac{2\lambda n}{d}$.

\begin{claim}\label{claim:verify-sorting-network}
The following properties hold:

\begin{enumerate}
    \item $G'$ is $m$-joined;

    \item $I(X\cup Y)$ is $(D,m)$-extendable in $G'$.
\end{enumerate}

\end{claim}

\begin{proof}[Proof of Claim \ref{claim:verify-sorting-network}]
Recall that by \ref{P2}, $G'$ is an $(\frac{n}{5},(1\pm2\gamma)\frac{d}{5}, \frac{6\lambda}{5})$-graph. Corollary \ref{cor:expander} implies that $G'$ is $m$-joined, so we are left to prove (2). In fact, since $I(X\cup Y)$ is an empty graph, $\Delta(I(X\cup Y))=0\leq D$. Also, recall that by \ref{P1}, for every vertex $v\in V$, $\deg_G(v,R_1)\geq (1-2\gamma)\frac{d}{5}\geq \frac{d}{6}$. Since $G$ is an $(n,(1\pm\gamma)d,\lambda)$-graph, by Lemma \ref{lem:expansion}, we have that for any subset $U\subseteq V(G')$ of size $1\leq |U|\leq 2m\leq \frac{4\lambda n}{d}$, 
\[
|N_{G'}(U)\setminus (X\cup Y)|\geq |N_{G}(U)\cap R_1|\geq \frac{d}{700\lambda}|U|=D|U|.
\]
Thus, by Proposition \ref{prop:weakerextendability}, $I(X\cup Y)$ is $(D,m)$-extendable in $G'$. This completes the proof.
\end{proof}

Now, let $C>0$ be the constant in Lemma \ref{lemma:find_sorting_network_in_expander}, and let $\ell\coloneqq\lfloor C\log^3(\frac{n}{5})\rfloor$. Since $D\geq 100$ and $m\leq \frac{n/5}{100D}$, by applying Lemma \ref{lemma:find_sorting_network_in_expander} on $G'$ with $\frac{n}{5},\frac{\log^4 n}{5\log^3(n/5)},X,Y$ in place of $n,K,V_1,V_2$, we can find a subgraph $S_{res}\subseteq G'=G[X\cup Y\cup R_1]$ with $|V(S_{res})|=k\ell$ satisfying the conclusion of the lemma. 

Assuming without loss of generality that $|R_1\cup R_2|$ is divisible by $k$, we further partition $R_1\cup R_2$ into $k$-sets, each of which has a small intersection with $S_{res}$, and every pair of such sets induces a bipartite expander. For convenience, we will let $V_1\coloneqq X$ and $V_t\coloneqq Y$ for the remainder of the proof, where $t\coloneqq \frac{n-2k}{k}$. 

\begin{claim}\label{claim:preparation-long-paths}
There exists a partition
\[
R_1\cup R_2=V_2\cup\ldots\cup V_{t-1}
\]
into $k$-sets such that the following properties hold:
\begin{enumerate}[label=(Q\arabic*)]
    \item\label{Q1} for each $2\leq i\leq t-1$, we have $|V_i\setminus V(S_{res})|\geq k-\frac{2Ck}{\log n}$;
    \item\label{Q2} for each $2\leq i\leq t-1$, $|V_{i,1}|=(1\pm\gamma)\frac{k}{5}$ and $|V_{i,2}|=(1\pm\gamma)\frac{4k}{5}$, where $V_{i,1}\coloneqq V_i\cap R_1$ and $V_{i,2}\coloneqq V_i\cap R_2$;

    \item \label{Q3} for each $2\leq i\leq t-1$ and for every vertex $v\in V$, we have  
    $$\deg(v,V_{i,1})=(1\pm5\gamma)\frac{dk}{5n} \textrm{ and } \deg(v,V_{i,2})=(1\pm5\gamma)\frac{4dk}{5n};$$

    \item \label{Q4} for each distinct $i,j\in[t]$, the subgraph \[
    G[V_i,V_j]\text{ is a }\left(2k,(1\pm5\gamma)\frac{2dk}{n},\frac{18\lambda k}{n}\right)\text{-bipartite expander};
    \]

    \item \label{Q5} for each distinct $i,j\in[t]$, the bipartite subgraph  %letting $V_{i,1}\coloneqq V_i\cap R_1$ and $V_{i,2}\coloneqq V_i\cap R_2$, 
    \[
    G[V_{i,1},V_{j,1}]\text{ is a }\left(|V_{i,1}|+|V_{j,1}|,(1\pm5\gamma)\frac{2dk}{5n},\frac{18\lambda k}{n}\right)\text{-bipartite expander,}
    \]
    where $V_{1,1}\coloneqq X_1$ and $V_{t,1}\coloneqq Y_1$.
\end{enumerate}
\end{claim}

\begin{proof}[Proof of Claim \ref{claim:preparation-long-paths}]
Let $
R_1\cup R_2=V_2\cup\ldots\cup V_{t-1}$ 
be a uniformly random partition into $k$-sets. We wish to prove that each property among \ref{Q1}--\ref{Q5} holds whp.

\medskip

{\em Property~\ref{Q1}.} Recall that $|S_{res}|=k\ell=\frac{n}{\log^4n}\cdot C\log^3(\frac{n}{5})$. Since for each $2\leq i\leq t-1$, $V_i$ is a uniformly random $k$-set, we have that $\mathbb{E}[|V_i\cap S_{res}|]=\frac{k}{\log^4n}\cdot C\log^3(\frac{n}{5})$. Thus, by Chernoff's bounds and the union bound,
\[
\Pr{|V_i\cap S_{res}|\geq \frac{2Ck}{\log n}\text{ for some }2\leq i\leq t-1}\leq te^{-\Theta\left(\frac{Ck}{\log n}\right)}=o(1).
\]
Therefore, property \ref{Q1} holds with probability $1-o(1)$.

\medskip

{\em Property~\ref{Q2}.} Recall that $|R_1|=\frac{n}{5}-2k$ and $|R_2|=\frac{4n}{5}$. Since $V_i$ is chosen as a uniformly random $k$-set, we have $\mathbb{E}[|V_{i,1}|]=\frac{k}{5}-\frac{2k^2}{n}$ and $\mathbb{E}[|V_{i,2}|]=\frac{4k}{5}$. By applying Chernoff's bounds and the union bound, we get
\[
\Pr{|V_{i,1}|\neq (1\pm\gamma)\frac{k}{5}\text{ or }|V_{i,2}|\neq (1\pm\gamma)\frac{4k}{5}\text{ for some }2\leq i\leq t-1}\leq te^{-\omega(\log n)}=o(1),
\]
where the last inequality holds since $k=\omega(\log n)$. Thus, property \ref{Q2} holds with probability $1-o(1)$. For the remainder of the proof, we will condition on property \ref{Q2}. %For the remainder of the proof, we condition on property \ref{Q2}.

\medskip

{\em Property~\ref{Q3}.} Recall from property \ref{P1} that for every vertex $v\in V$, we have $\deg(v,R_1)=(1\pm2\gamma)\frac{d}{5}$. Additionally, by \ref{Q2}, for each $2\leq i\leq t-1$, $V_{i,1}$ is a random subset of size $(1\pm\gamma)\frac{k}{5}$. Now, conditioning on  $|V_{i,1}|=a=(1\pm\gamma)\frac{k}{5}$, $\deg(v,V_{i,1})$ is a hypergeometric random variable with expectation $(1\pm2\gamma)\frac{da}{n}=(1\pm4\gamma)\frac{dk}{5n}$. Thus, fixing $v\in V$ and $2\leq i\leq t-1$, Chernoff's bounds imply that 
\[
\Pr{\deg(v,V_{i,1})\neq \left(1\pm5\gamma\right)\frac{dk}{5n}\,\Big|\,|V_{i,1}|=a}
\leq e^{-\omega(\frac{dk}{n})}.
\]
By the law of total probability and the union bound, we obtain that
\[
\begin{aligned}
&\Pr{\exists v\in V,\exists 2\leq i\leq t-1,\deg(v,V_{i,1})\neq \left(1\pm5\gamma\right)\frac{dk}{5n}}\\
\leq & nt\sum_{a=(1-\gamma)\frac{k}{5}}^{(1+\gamma)\frac{k}{5}}\Pr{\deg(v,V_{i,1})\neq \left(1\pm5\gamma\right)\frac{dk}{5n}\,\Big|\,|V_{i,1}|=a}\cdot\Pr{|V_{i,1}|=a}\\
\leq & nte^{-\omega(\frac{dk}{n})}\sum_{a=(1-\gamma)\frac{k}{5}}^{(1+\gamma)\frac{k}{5}}\Pr{|V_{i,1}|=a}\\
\leq & nte^{-\omega(\frac{dk}{n})}\\
=& o(1).
\end{aligned}
\]

Therefore, with an analogous proof for $\deg(v,V_{i,2})$, property \ref{Q3} holds with probability $1-o(1)$.

%\red{AF:you need to say few more words here about the propoerties P something. At first sight, to me it looked like a typo and that you meant to write Q3 and Q4.}
%By \ref{P3}, \ref{P4} and a union bound, property \ref{Q4} holds with probability at least 
% \[
% 1-\binom{t}{2}n^{1/56}=1-o(1).
% \]

%\blue{DM: The proof with better wording might like below: }

\medskip

{\em Property~\ref{Q4}.} First, recall that by \ref{P5}, $G[X,Y]$ is a $\left(2k,(1\pm2\gamma)\frac{2dk}{n},\frac{12\lambda k}{n}\right)$-bipartite expander. Next, recall that by \ref{P3}, for two random disjoint $k$-sets $V_i,V_j\subseteq R_1\cup R_2$ chosen uniformly at random, with probability at least $1-n^{-1/7}$, $G[V_i,V_j]$  is a $\left(2k,(1\pm4\gamma)\frac{2dk}{n},\frac{18\lambda k}{n}\right)$-bipartite expander. Thus, by the union bound, with probability at least $1-\binom{t-2}{2}n^{-1/7}$, for each pair $2\leq i<j\leq t-1$, $G[V_i,V_j]$  is a $\left(2k,(1\pm4\gamma)\frac{2dk}{n},\frac{18\lambda k}{n}\right)$-bipartite expander.

Finally, recall that by \ref{P4}, for a random $k$-sets $V_i\subseteq R_1\cup R_2$ chosen uniformly at random, with probability at least $1-n^{-1/28}$, $G[X,V_i]$ and $G[V_i,Y]$ are $\left(2k,(1\pm2\gamma)\frac{2dk}{n},\frac{12\lambda k}{n}\right)$-bipartite expanders. Again, by the union bound, with probability at least $1-(t-2)n^{-1/28}$, for each $2\leq i\leq t-1$, $G[X,V_i]$ and $G[V_i,Y]$ are $\left(2k,(1\pm4\gamma)\frac{2dk}{n},\frac{18\lambda k}{n}\right)$-bipartite expanders.

Therefore, in total, \ref{Q4} holds with probability at least 
\[
1-\binom{t-2}{2}n^{-1/7}-(t-2)n^{-1/28}=1-o(1).
\]

For the remainder of the proof, we will condition on property \ref{Q3} and \ref{Q4}. 
%In the rest of the proof, it is enough to condition on property \ref{Q3} and \ref{Q4}.

\medskip

% \blue{DM: All below are modified to be easy to read.}

{\em Property~\ref{Q5}.} Recall that by \ref{P1} and \ref{Q3}, we have that for each $i\in[t]$ and for every vertex $v\in V$, $\deg(v,V_{i,1})=(1\pm5\gamma)\frac{dk}{5n}$. Also, recall that \ref{Q4} implies that $s_2(G[V_i\cup V_j])\leq \frac{18\lambda k}{n}$. Thus, by the Interlacing Theorem for singular values (Theorem \ref{thm:Interlacing-Theorem}), $s_2(G[V_{i,1}\cup V_{j,1}])\leq s_2(G[V_i\cup V_j])\leq \frac{18\lambda k}{n}$. Therefore,
\[
    G[V_{i,1},V_{j,1}]\text{ is a }\left(|V_{i,1}|+|V_{j,1}|,(1\pm5\gamma)\frac{2dk}{5n},\frac{18\lambda k}{n}\right)\text{-bipartite expander.}
    \]

\medskip

All in all, with positive probability all properties \ref{Q1}--\ref{Q5} hold, which guarantees a partition $R_1\cup R_2=V_2\cup\ldots\cup V_{t-1}$ satisfying all desired properties. This completes the proof.
\end{proof}

\subsection{Finding vertex-disjoint paths}

We pick a partition $R_1\cup R_2=V_2\cup\ldots\cup V_{t-1}$ as in Claim \ref{claim:preparation-long-paths}, and let $U_i\coloneqq V_i\setminus V(S_{res})$ for each $2\leq i\leq t-1$. Also, denote $U_1\coloneqq V_1$ %\red{do you really need to label $V_1$ as $U_1$? two lines below you just write $V_1$. Also, why didn't you label $V_t$ as a $U_2$ for example?}
for simplicity in the iterative process defined below. By reordering if necessary, without loss of generality, we may assume that $|U_1|\geq |U_2|\geq \ldots \geq |U_{t-1}|,$ and let $n_1\geq n_2\geq \ldots \geq n_{t-1}$ be the corresponding sizes of the $|U_i|$s. We wish to find vertex-disjoint paths covering all the vertices in $\bigcup_{i=1}^{t-1} U_i\cup V_t$, where each path has one endpoint in $U_1=X$ and the other in $V_t=Y$. Note that if all the $n_i$s were the same, then we could simply find perfect matchings between each pair $U_i$ and $U_{i+1}$. Since the $n_i$s are not the same, we will first use Lemma \ref{lem:small-matching} to find vertex-disjoint matchings $M_i$s, where $M_i$ is a matching of size $n_i-n_{i+1}$ between $U_i$ and $V_t$. Then, noticing that the remainder of each two consecutive parts have the same size, we will apply Lemma \ref{perfect-matching-between-random-subsets} to find a perfect matching $N_i$ between the remainder of each two consecutive parts. Now, concatenating all the matchings $M_i$s and $N_i$s will give the desired vertex-disjoint paths (see Figure \ref{fig:matchings} for illustration). 

Initially, let $V_t'\coloneqq V_t$. First, we find a matching $M_1$ of size $n_1-n_2$ between $U_1\cap V_{1,1}$ and $V_t'\cap V_{t,1}$. Recall that we have $n_1-n_2\leq \frac{2Ck}{\log n}$ by \ref{Q1}. Also, recall that by \ref{Q5}, 
\[
G[V_{1,1},V_{t,1}]\text{ is a }\left(|V_{1,1}|+|V_{t,1}|,(1\pm5\gamma)\frac{2dk}{5n},\frac{18\lambda k}{n}\right)\text{-bipartite expander.}
\]
Thus, by Lemma \ref{lem:small-matching}, there exists a matching $M_1$ of size $n_1-n_{2}$ in $G$ between $U_1\cap V_{1,1}$ and $V_t'\cap V_{t,1}$. Let $U_1'\coloneqq U_1\setminus V(M_1)$ and $V_t'\coloneqq V_t'\setminus V(M_1)$. Now, we have that $|U_1'|=|U_2|=n_2$, which is necessary to find a perfect matching between $U_1'$ and $U_2$. Also, $|V_t'|=n_2$.

Next, we find a matching $M_2$ of size $n_2-n_3$ between $U_2\cap V_{2,1}$ and $V_t'\cap V_{t,1}$. Recall that by \ref{Q1}, we have 
\[
|V_{2,1}\setminus U_2|+(n_2-n_3)=|V_{t,1}\setminus V_t'|+(n_2-n_3)=n_1-n_3\leq \frac{2Ck}{\log n}
\]
Also, recall that by \ref{Q5}, 
\[
G[V_{2,1},V_{t,1}]\text{ is a }\left(|V_{2,1}|+|V_{t,1}|,(1\pm5\gamma)\frac{2dk}{5n},\frac{18\lambda k}{n}\right)\text{-bipartite expander.}
\]
Thus, by Lemma \ref{lem:small-matching}, there exists a matching $M_2$ of size $n_2-n_{3}$ in $G$ between $U_2\cap V_{2,1}$ and $V_t'\cap V_{t,1}$. Let $U_2'\coloneqq U_2\setminus V(M_2)$ and $V_t'\coloneqq V_t'\setminus V(M_2)$. Now, we have that $|U_2'|=|U_3|=n_3$, which is necessary to find a perfect matching between $U_2'$ and $U_3$. Also, $|V_t'|=n_3$. Continuing in this fashion, we can and do find vertex-disjoint matchings $M_i$s between $U_i$ and $V_t$ for each $i\in [t-2]$. And crucially, we have that $|U_i'|=|U_{i+1}|=n_{i+1}$ for each $i\in[t-2]$, and $|U_{t-1}|=|V_{t}'|=n_{t-1}$. In the rest of proof, we let $U_{t-1}'\coloneqq U_{t-1}$ and $U_t\coloneqq V_t'$ for simplicity of the following iterative process.

Now, we find a perfect matching between $U_i'$ and $U_{i+1}$ for each $i\in[t-1]$. To do so, we first prove that the induced bipartite subgraph $G[U_i',U_{i+1}]$ is a bipartite expander. Recall that for each $i\in[t-1]$, $G[V_i,V_{i+1}]$ is a $\left(2k,(1\pm5\gamma)\frac{2dk}{n},\frac{18\lambda k}{n}\right)$-bipartite expander by \ref{Q4}. So for every vertex $v\in V_i\cup V_{i+1}$,
\[
\deg(v,U_i')\leq \deg(v,V_i)\leq (1+5\gamma)\frac{dk}{n}\leq \left(1+\frac{1}{8}\right)\frac{9dk}{10n},
\]
and 
\[
\deg(v,U_{i+1})\leq \deg(v,V_{i+1})\leq (1+5\gamma)\frac{dk}{n}\leq \left(1+\frac{1}{8}\right)\frac{9dk}{10n}.
\]
Also, recall that  by \ref{Q3}, for every vertex $v\in V_i\cup V_{i+1}$, $\deg(v,V_{i,2}),\deg(v,V_{i+1,2})=(1\pm5\gamma)\frac{4dk}{5n}$. Since $V_{i,2}\subseteq U_i'$ and $V_{i+1,2}\subseteq U_{i+1}$ for each $i\in[t-1]$, this implies that
\[
\deg(v,U_i')\geq \deg(v,V_{i,2})\geq (1-5\gamma)\frac{4dk}{5n}\geq \left(1-\frac{1}{8}\right)\frac{9dk}{10n},
\] 
and 
\[
\deg(v,U_{i+1})\geq \deg(v,V_{i+1,2})\geq (1-5\gamma)\frac{4dk}{5n}\geq \left(1-\frac{1}{8}\right)\frac{9dk}{10n},
\] 
where we used $\gamma\leq 1/400$ in the last step of both inequalities. Finally, by the Interlacing Theorem for singular values (Theorem \ref{thm:Interlacing-Theorem}), $s_2(G[U_i'\cup U_{i+1}])\leq s_2(G[V_i\cup V_{i+1}])\leq \frac{18\lambda k}{n}$. Therefore, 
\[
G[U_i', U_{i+1}]\text{ is a }\left(|U_i'|+|U_{i+1}|,\left(1\pm\frac{1}{8}\right)\frac{9dk}{5n},\frac{18\lambda k}{n}\right)\text{-bipartite expander}.
\]

Thus, since $\frac{18\lambda k}{n}\leq \frac{1}{200}\cdot\frac{9dk}{5n}$, by Lemma \ref{perfect-matching-between-random-subsets}, there exists a perfect matching $N_i$ in $G$ between $U_i'$ and $U_{i+1}$. Now, by concatenating all the matchings $M_i$s and $N_i$s together, we obtain vertex-disjoint paths $P_1,\ldots,P_k$ covering $\bigcup_{i=1}^{t-1} U_i\cup V_t$, where each path has one endpoint in $U_1=X$ and the other in $V_t=Y$.

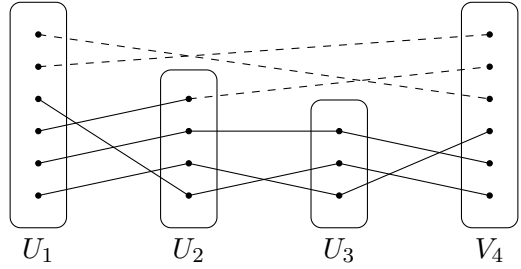
\begin{figure}
\begin{center}

    \begin{tikzpicture}
    %\centering
    %%%%%%%%%%%%    first row    %%%%%%%%%%%%%%%%
    \draw[rounded corners] (3.5, 6) rectangle (4.25, 3) {}
        node at (3.875,2.7){$U_1$};
    \draw[rounded corners] (5.5, 5.1) rectangle (6.25, 3) {}
        node at (5.875,2.7){$U_2$}; 
    \draw[rounded corners] (7.5, 4.7) rectangle (8.25, 3) {}
        node at (7.875,2.7){$U_3$};
    \draw[rounded corners] (9.5, 6) rectangle (10.25, 3) {}
        node at (9.875,2.7){$V_4$};
   
    %%%%%%%%%%%%    vertices in X    %%%%%%%%%%%%%%%%
    \foreach \i [evaluate=\i as \x using \i*3/7+3] in {1,2,...,6}{
          \filldraw(3.875,\x) circle[radius=1pt];
      }
    
     %%%%%%%%%%%%    vertices in Y    %%%%%%%%%%%%%%%%
    \foreach \i [evaluate=\i as \x using \i*3/7+3] in {1,2,...,4}{
          \filldraw(5.875,\x) circle[radius=1pt];
      }
    %%%%%%%%%%%%    vertices in R    %%%%%%%%%%%%%%%%
    
     %%%%%%%%%%%%    vertices in Y    %%%%%%%%%%%%%%%%
    \foreach \i [evaluate=\i as \x using \i*3/7+3] in {1,2,3}{
          \filldraw(7.875,\x) circle[radius=1pt];
      }

    \foreach \i [evaluate=\i as \x using \i*3/7+3] in {1,2,...,6}{
          \filldraw(9.875,\x) circle[radius=1pt];
      }

    %%%%%%%%%%%%    draw lines V1 and V2   %%%%%%%%%%%%%%%%
    \draw (3.875,4*3/7+3)--(5.875,1*3/7+3);
    \draw (3.875,1*3/7+3)--(5.875,2*3/7+3);
    \draw (3.875,2*3/7+3)--(5.875,3*3/7+3);
    \draw (3.875,3*3/7+3)--(5.875,4*3/7+3);
    %%%%%%%%%%%%    draw lines V2 and V3   %%%%%%%%%%%%%%%%
    \draw (5.875,1*3/7+3)--(7.875,2*3/7+3);
    \draw (5.875,2*3/7+3)--(7.875,1*3/7+3);
    \draw (5.875,3*3/7+3)--(7.875,3*3/7+3);
    %%%%%%%%%%%%    draw lines V3 and V4   %%%%%%%%%%%%%%%%
    \draw (7.875,1*3/7+3)--(9.875,3*3/7+3);
    \draw (7.875,2*3/7+3)--(9.875,1*3/7+3);
    \draw (7.875,3*3/7+3)--(9.875,2*3/7+3);
    %%%%%%%%%%%%    draw lines rest   %%%%%%%%%%%%%%%%
    \draw[dashed] (3.875,5*3/7+3)--(9.875,6*3/7+3);
    \draw[dashed] (3.875,6*3/7+3)--(9.875,4*3/7+3);
    \draw[dashed] (5.875,4*3/7+3)--(9.875,5*3/7+3);
\end{tikzpicture}
\caption{The figure is an example for connecting the matchings into vertex-disjoint paths when $t=4$. The dashed lines represent matchings $M_i$s, and the straight lines represent matchings $N_i$s.}\label{fig:matchings}
    
\end{center}
\end{figure}

\subsection{Closing paths into a cycle}

Let $a_i\in X$ and $b_i\in Y$ be the endpoints of the path $P_i$. Now, Lemma \ref{lemma:find_sorting_network_in_expander} implies that $S_{res}$ has a path-factor $Q_1,\dots,Q_k$ such that $Q_i$ connects $b_{i}$ and $a_{i+1}$, where $a_{t+1}\coloneqq a_1$. Therefore, $P_1Q_1P_2Q_2\ldots P_kQ_k$ is a Hamilton cycle of $G$. This completes the proof.

\end{proof}

\subsection*{Acknowledgement}
Part of this work was done when Jie Han visited the math department of UCI in spring 2023, and he would like to thank the department for the hospitality. 
We would like to thank an anonymous referee whose careful reading and valuable suggestions
helped us improve the paper considerably.
In particular, we thank the referee for suggesting the result from~\cite{hyde2023spanning} which replaces the absorption method and simplifies our proof significantly.

\bibliographystyle{abbrv}
	\bibliography{main}

\appendix
\section{Linear algebra background}\label{s: linear algebra}
%===============

In this section we collect some standard tools from linear algebra.

The following theorem provides a convenient tool for computing/bounding eigenvalues of a real symmetric matrix (see for example Theorem 2.4.1 in \cite{brouwer2011spectra}).

\begin{thm}[Courant-Fischer Minimax Theorem]\label{thm:minimax}
Let $A$ be a symmetric $n\times n$ matrix with eigenvalues $\lambda_1\geq\lambda_2\geq\dots\geq \lambda_n$. Then,
\[
\lambda_k=\max_{\dim(U)=k}\min_{\vx\in U\setminus\{\mathbf{0}\}}\frac{\vx^\tran A\vx}{\vx^\tran\vx}=\min_{\dim(U)=n-k+1}\max_{\vx\in U\setminus\{\mathbf{0}\}}\frac{\vx^\tran A\vx}{\vx^\tran\vx}.
\]
\end{thm}

Since the notion of eigenvalues is undefined for non-square matrices, it would be convenient for us to work with singular values which are defined for all matrices (see Definition \ref{def:singular-value}). The following theorem proved by Thompson \cite{thompson1972principal} is useful when one wants to obtain non-trivial bounds on the singular values of submatrices.

\begin{thm}[Interlacing Theorem for singular values]\label{thm:Interlacing-Theorem}
Let $A$
be an $m\times n$ matrix and let 
$$\alpha_1\geq\alpha_2\geq\ldots\geq\alpha_{\min\{m,n\}}$$
be its singular values. Let $B$ be any $p\times q$ submatrix of $A$ and let $$\beta_1\geq\beta_2\geq\ldots\geq\beta_{\min\{p,q\}}$$ be its singular values. Then
\[
\begin{aligned}
\alpha_i\geq\beta_i,\quad\quad\quad\quad\quad\quad\:\:&\text{for }i=1,2,\ldots,\min\{p,q\},\\
\beta_i\geq \alpha_{i+(m-p)+(n-q)},\quad&\text{for }i\leq\min\{p+q-m,p+q-n\}.
\end{aligned}
\]
\end{thm}

One of the most commonly used tools in linear algebra is the singular value decomposition. We need a slightly stronger version of it, which almost immediately follows from the standard proof:

\begin{thm}[Singular value decomposition]
\label{thm:SVD}
  Let $M$ be an $m\times n$ matrix with rank $r$. Let $s_1\geq s_2\geq \dots\geq s_r$ be all the positive singular values of $M$. Let $\vu_1\in \R^m$ and $\vv_1\in \R^n$ be unit vectors such that $M\vv_1=s_1\vu_1$. Then we can find an orthonormal bases $\{\vu_1,\dots,\vu_m\}$ of $\R^m$ and $\{\vv_1,\dots,\vv_n\}$ of $\R^n$ with $\vu_1$ and $\vv_1$ as above, and such that
  \[M=\sum_{j=1}^rs_j\vu_j\vv_j^\tran.\] 
In particular, this equality implies that $M\vv_j=s_j\vu_j$ for $j=1,\ldots,r$ and $M\vv_j=\mathbf{0}$ for $j>r$.
\end{thm}

We will also make use of the following simple corollary of the above theorem, which proof is included for completion.

\begin{lem}[Best low-rank approximation]\label{lem:low-rank-approximation}
Let $A$ be an $m \times n$
matrix. Then
\[
s_2(A)=\min_B\:\norm{A-B},
\]
where the minimum is over all rank-one $m\times n$ matrices $B$, and
$\|{\mkern 2mu\cdot\mkern 2mu}\|$ denotes the operator norm.

Moreover, the minimum is attained by $B = s_1(A)\vu_1\vv^\tran_1$, where 
$\vv_1\in \R^n$ and $\vu_1\in \R^m$ are any unit vectors such that $A\vv_1=s_1(A)\vu_1$.
\end{lem}

\begin{proof}
Let $A$ be an $m \times n$ 
matrix with rank $r$. Let $s_1\geq s_2\geq\ldots\geq s_r$ be all positive singular values of $A$, and let $\vv_1\in \R^n$ and $\vu_1\in \R^m$ be unit vectors such that $A\vv_1=s_1\vu_1$. By Theorem \ref{thm:SVD}, there exist orthonormal bases $\{\vv_1,\dots,\vv_n\}$ of $\R^n$ and $\{\vu_1,\dots,\vu_m\}$ of $\R^m$, such that \[A=\sum_{j=1}^rs_j\vu_j\vv_j^\tran.
\]

First, note that $B=s_1\vu_1\vv_1^\tran$ is a rank-one matrix that satisfies 
\[
\norm{A-B}
=\norm[2]{\sum_{j=2}^r s_j\vu_j\vv_j^\tran}
=s_2.
\]
Therefore, to finish the proof, it suffices to show $s_2\leq \|A-B\|$ for every rank-one $m\times n$ matrix $B$. We can express such a matrix as $B=\mathbf{x}\mathbf{y}^\tran$ for some nonzero vectors $\mathbf{x} \in \mathbb{R}^m$ and $\mathbf{y} \in \R^n$. Next, we can find a nontrivial linear combination $\mathbf{w}=a\vv_1+b\vv_2$ such that $\langle \mathbf{y},\mathbf{w} \rangle = \mathbf{0}$; this implies $B\mathbf{w}=\mathbf{x}(\mathbf{y}^\tran\mathbf{w})=0$. Without loss of generality, we can scale $\mathbf{w}$ so that $\|\mathbf{w}\|=1$, or equivalently, $a^2+b^2=1$. Therefore, 
\[
\norm{A-B}^2\geq \norm{(A-B)\mathbf{w}}_2^2=\norm{A\mathbf{w}}_2^2=a^2s_1^2+b^2s_2^2\geq s_2^2.
\]
This completes the proof.
\end{proof}

Finally, we state the chain rule for singular values, which is used in the proof of Corollary~\ref{cor: singular values normalized}.

\begin{lem}[Chain rule for singular values] \label{lem: chain rule}
    Let $A,B,C$ be $n \times n$ matrices. Then 
    \[
    s_k(ABC) \le \norm{A}  \norm{B}s_k(C)
    \quad \text{for all } k \in [n].
    \]
\end{lem}
    
\begin{proof}
    First assume that $C=I$. By the Minimax Theorem~\ref{thm:minimax}, we have 
    \[
    s_k(AB) = \min_{\dim(U)=n-k+1}\max_{\vx \in S(U)} \norm{AB\vx}_2,
    \]
    where $S(U)$ denotes the set of all unit vectors in $U$. Since $\norm{AB\vx}_2 \le \norm{A} \norm{B\vx}_2$, it follows that $s_k(AB) \le \norm{A} s_k(B)$.
    This argument also yields $s_k(BC) \le s_k(B) \norm{C}$ once we notice that $s_k(BC)=s_k(C^\tran B^\tran)$. Combining these two bounds, we complete the proof.
\end{proof}

\end{document}